\documentclass[11pt]{article}

\usepackage[margin=1in]{geometry}

\usepackage[T1]{fontenc}
\usepackage[utf8]{inputenc}

\usepackage{amsmath,amssymb,amsfonts,amsthm,mathtools}
\RequirePackage[numbers,sort&compress]{natbib}
\usepackage{graphicx}
\usepackage{booktabs}
\usepackage{multirow}
\usepackage{algorithm}
\usepackage{algorithmic}
\usepackage{enumitem}
\usepackage{macro}
\usepackage{bbm}
\usepackage{multirow}
\usepackage{booktabs}
\usepackage{mathtools}
\usepackage{commath}
\usepackage{algorithm}
\usepackage{algorithmic}

\newcommand{\qa}{\hat{q}}
\newcommand{\qt}{\tilde{q}}
\newcommand{\qto}{\tilde{q}_l}
\newcommand{\qtt}{\tilde{q}_u}
\newcommand{\tqto}{\tilde{q}_u}
\newcommand{\tqtt}{\tilde{q}_l}
\newcommand{\Do}{\mathcal{D}_1}
\newcommand{\Dt}{\mathcal{D}_2}
\newcommand{\hf}{\tilde{F}}
\newcommand{\tf}{\hat{F}}
\newcommand{\fe}{F_\epsilon}
\newcommand{\fs}{F_S}
\newcommand{\fme}{F_{|\epsilon|}}
\newcommand{\hg}{\hat{\gamma}_1}
\newcommand{\G}{\mathcal{G}}
\newcommand{\La}{\mathcal{\hat{L}}}
\newcommand{\Lt}{\mathcal{\tilde{L}}}
\newcommand{\R}{\mathbb{R}}

\usepackage[numbers,sort&compress]{natbib}

\usepackage[
    colorlinks=true,
    citecolor=blue,
    linkcolor=blue,
    urlcolor=blue
]{hyperref}

\newtheorem{theorem}{Theorem}[section]

\theoremstyle{definition}
\newtheorem{definition}[theorem]{Definition}

\newtheorem{remark}[theorem]{Remark}

\title{\bfseries On Optimal Data Splitting for Split Conformal Prediction}

\author{
Sayan Das\thanks{Department of Statistics and Data Science, Washington University in St. Louis, St. Louis, MO 63130, USA}
\and
Bahram Yaghooti\thanks{Department of Electrical and Systems Engineering, Washington University in St. Louis, St. Louis, MO 63130, USA}
\and
Todd A. Kuffner$^*$
\and
Soumendra N. Lahiri$^*$
}

\date{}

\begin{document}

\maketitle

%%%%%%%%%%%%%%%%%%%%%%%%%%%%%%%%%%%%%%%%%%%%%%%%%%%%%%%%%%%%%%%%%%%%%%%%%%%%%%%
% Abstract
%%%%%%%%%%%%%%%%%%%%%%%%%%%%%%%%%%%%%%%%%%%%%%%%%%%%%%%%%%%%%%%%%%%%%%%%%%%%%%%
\vspace{0.4cm}

\begin{abstract}
Conformal prediction and its variants, including the split conformal prediction, provide a distribution-free framework for uncertainty quantification by constructing prediction intervals or sets with finite-sample coverage guarantees. The statistical efficiency of these intervals depends critically on how the data are split into training and calibration samples. Despite its practical importance, a principled characterization of the training-calibration split that minimizes prediction interval length while maintaining coverage has remained largely unresolved. In this paper, we develop a theoretical framework for optimal data splitting in split conformal prediction. We first analyze the problem in a general setting and derive analytical characterizations of the length-optimal split ratio under both symmetric and asymmetric regimes. We then show how the general results specialize to several commonly used regression settings, including linear regression, nonparametric regression, and neural networks, thereby demonstrating the scope of the framework. We also describe a data-based method for selecting the optimal proportion. Our analysis clarifies how model-related features govern the optimal allocation of samples between training and calibration and provides principled guidance for constructing shorter prediction intervals. Experiments on both synthetic and real-world datasets demonstrate the applicability of the proposed methodology across a variety of practical scenarios.
\end{abstract}

\vspace{0.5em}

\noindent
\textbf{Keywords:}
Conformal prediction;
optimal sample allocation;
asymptotic optimality;
distribution-free inference.

\vspace{0.3em}

\noindent
\textbf{MSC2020:}
Primary 62G15, 62G20;
Secondary 62G08.

%%%%%%%%%%%%%%%%%%%%%%%%%%%%%%%%%%%%%%%%%%%%%%%%%%%%%%%%%%%%%%%%%%%%%%%%%%%%%%%
\section{Introduction}\label{section:introduction}
Conformal prediction is a distribution-free framework for uncertainty quantification that constructs prediction intervals or sets with finite-sample marginal coverage guarantees, assuming only data exchangeability~\cite{vovk2005algorithmic}. Its key advantage lies in providing rigorous, model-agnostic calibration through a post-hoc mechanism that converts point predictions into statistically valid confidence intervals or sets, without modifying the underlying learning algorithm~\cite{angelopoulos2023conformal}. Classical conformal prediction, also known as full conformal prediction, guarantees exact coverage but requires retraining the model for each possible output, limiting its practical applicability in large-scale settings due to substantial computational complexity.

To address these computational challenges, several variants have emerged. Among these, inductive conformal prediction, also known as split conformal prediction, stands out for its practicality: it partitions data into (proper) training- and calibration-subsets and achieves valid coverage with a single model fit~\cite{papadopoulos2002inductive}. Other extensions, such as cross-conformal prediction~\cite{vovk2015cross} and Jackknife+~\cite{barber2021predictive}, enhance the stability and sharpness of prediction intervals through resampling techniques. Additionally, online conformal prediction extends the framework to sequential or streaming settings, enabling real-time uncertainty quantification~\cite{gibbs2024conformal,angelopoulos2024online}.

Beyond computational efficiency, recent research has substantially expanded the applicability of conformal prediction. Adaptive techniques~\cite{zaffran2022adaptive}, blockwise calibration methods~\cite{xu2023conformal}, and covariate-adjusted conformal inference~\cite{tibshirani2019conformal} have been developed to address challenges such as non-exchangeable data and distributional shift. Further efforts focus on achieving stronger forms of validity, including class-conditional coverage guarantees~\cite{ding2023class}. In addition, 
~\cite{dhillon2024expected} 
examines the relationship between calibration set size and predictive accuracy under computational constraints, underscoring the importance of principled data allocation strategies.

Despite split conformal prediction's theoretical guarantees and widespread adoption across a broad range of application domains, including medical diagnostics~\cite{sreenivasan2025conformal}, natural language processing~\cite{campos2024conformal}, computer vision~\cite{melki2023group,angelopoulos2021uncertainty}, finance~\cite{pasche2025extreme}, medical imaging~\cite{lu2022fair}, robotics and autonomous vehicles~\cite{luo2024sample}, reinforcement learning~\cite{gan2025conformal}, and large language models~\cite{cherian2024large}, determining optimal split ratio and its impact on prediction interval length remains a major challenge. Prior work has analyzed the expected size of conformal prediction sets in both asymptotic and finite-sample regimes~\cite{lei2013distribution,sadinle2019least,vovk2016criteria,lei2014distribution,dhillon2024expected}. However, none of these studies provide procedures for deriving or achieving optimal-length prediction intervals, leaving open the question of how to systematically determine the split ratio that yields the smallest valid prediction interval.
% in practice.

Another growing line of work has begun to address this gap by explicitly developing methods for prediction interval length optimization. Several approaches focus on designing improved conformity scores that yield sharper and more informative sets under various forms of conditional validity~\cite{lei2018distribution,romano2019conformalized,deutschmann2024adaptive,feldman2021improving,romano2020classification,xie2024boosted}. Another direction leverages insights from level-set estimation, drawing on foundations from classical statistics and their adaptation to conformal prediction in~\cite{lei2011efficient,sadinle2019least}, to construct covariate-dependent thresholds that directly target shorter prediction sets while accommodating refined coverage notions. Complementing these developments, conformal training methods~\cite{bai2022efficient,cherian2024large,stutz2022learning} optimize the conformity score itself to reduce prediction-set size, whereas alternative formulations keep the score fixed and instead optimize adaptive, covariate-informed thresholds. More recently, a principled minimax framework for constructing optimal length prediction sets under various conditional coverage requirements has been introduced, further emphasizing the growing focus on prediction-set efficiency~\cite{kiyani2024length}.

Collectively, these advances mark an important shift toward methods that not only ensure proper coverage but also aim to make conformal prediction sets as compact and informative as possible. Nevertheless, split conformal prediction continues to face a critical limitation: its performance is highly sensitive to how the data are partitioned between training and calibration sets. In practice, most implementations rely on fixed or randomly chosen split ratios without principled justification, which can lead to suboptimal performance, especially in data-limited or imbalanced settings. This challenge introduces a fundamental trade-off: Allocating more data to training improves predictive accuracy, whereas increasing the size of the calibration set enhances the reliability of the uncertainty quantification. Yet, despite this inherent tension, no principled framework currently exists for selecting the optimal split ratio, leaving a significant gap in the practical deployment of split conformal prediction.

In this paper, we address the length-optimal data splitting problem in split conformal prediction by introducing a theoretical framework for selecting the optimal training-calibration sample size ratio. Our contributions are as follows:
\begin{itemize}
\item We derive analytical expressions that characterize the optimal training-calibration split ratio required to jointly optimize the prediction interval length while preserving nominal coverage accuracy. 
These results are obtained under general distributional assumptions, including both symmetric and asymmetric settings, and offer insights into the dependence of the optimal split on key distributional parameters.
\item We instantiate these theoretical findings within three representative classes of regression models, namely linear regression, nonparametric regression, and neural network-based regression, thereby illustrating the broad applicability and model-agnostic nature of the proposed framework. We further analyze the theoretical and empirical implications of our results in each case.
\item 
We also provide a general  data based method for  determining 
the optimal training-calibration  split ratio that is agnostic to the 
learning algorithm. Applied to the 
three specific classes of problems mentioned above, the proposed method provides practically useful and accurate 
recommendations of the split ratio.
\item We conduct extensive empirical evaluations on both synthetic and real-world datasets, highlighting the impact of principled data splitting on prediction interval length across diverse regression models. In the absence of prior work on optimal split selection, our experimental results underscore the importance of model-aware splitting strategies, demonstrating clear advantages over fixed split ratios while preserving finite-sample coverage guarantees.
\end{itemize}
Collectively, these results provide the first principled solution to a fundamental challenge in split conformal prediction, enhancing both the efficiency and reliability of uncertainty quantification in practical applications.

The remainder of the paper is organized as follows.
Section~\ref{section:background} reviews split conformal prediction and its coverage guarantees.
Section~\ref{section:main-results} presents our optimal splitting framework and theoretical results for various regression models.
Section~\ref{section:experimetns} reports empirical validation using synthetic and real-world datasets.
We conclude in Section~\ref{section:discussion} with key findings and directions for future work.

\section{Background on split conformal prediction}\label{section:background}
Conformal prediction provides a general, distribution-free framework for constructing predictive sets with finite-sample coverage guarantees~\cite{vovk2005algorithmic,shafer2008tutorial}. Let $(X_i,Y_i)_{i=1}^{n+1}$ be a random sample 
%i.i.d.\ observations 
from an unknown distribution $\mathbf{P}_{XY},$ where $X_i\in\mathcal{X}$ and $Y_i\in\mathcal{Y}$. Based on the first $n$ observations, the objective is to construct a measurable set $\mathcal{C}(X_{n+1})\subseteq\mathcal{Y}$ such that, for a prespecified miscoverage level $\alpha \in (0,1),$
\[
    \Pr\{Y_{n+1}\in \mathcal{C}(X_{n+1})\} \ge 1-\alpha,
\]
without imposing parametric assumptions on $\mathbf{P}_{XY}$. % beyond exchangeability.

In this paper, we focus on the regression setting, where $Y\in\mathbb{R}$ and
\[
    Y = f(X) + \epsilon,\qquad X\sim \mathbf{P}_X,\qquad 
    \epsilon \overset{\mathrm{iid}}{\sim} F_\epsilon,
\]
so that $\mathcal{C}(X_{n+1})$ is an interval.

\subsection{Split conformal prediction} Split conformal prediction divides the observed sample $(X_i,Y_i)_{i=1}^n$ into a \emph{proper training set}, say $\Do$, of size $n_1$ and a \emph{calibration set}, say $\Dt$, of size $n_2=n-n_1$. A regression estimator $\hat f_1 \equiv \hat f_{n_1}$ is fitted only on the proper training data $\Do$, while the calibration sample is used to compute conformity scores. The sample-splitting step ensures statistical independence between the calibration residuals, which is essential for achieving a finite-sample coverage guarantee~\cite{papadopoulos2002inductive,lei2018distribution}. The choice of $(n_1,n_2)$ determines the trade-off between estimator accuracy (improved by a larger $n_1$) and coverage accuracy (improved by a larger $n_2$). In later sections, we analyze this trade-off through expansions of the expected length of the prediction interval and quantify its potential impact on the coverage error, which must be balanced for optimal performance of the split conformal 
prediction method. 

\subsection{Symmetric conformal prediction}
In the standard regression setting, conformity scores are taken as absolute residuals~\cite{angelopoulos2023conformal}:
\[
    \hat\epsilon_i = |Y_i - \hat f_1(X_i)|, \qquad i \in \Dt.
\]
Let $\hat{q} \equiv \hat{q}_{n_2,1-\alpha}$ denote the 
%empirical $\lceil (n_2+1)(1-\alpha)\rceil$-quantile 
$\lceil (n_2+1)(1-\alpha)\rceil$-th
order statistic of $\{\hat\epsilon_i : i\in \Dt\}$ 
where, for $x\in \bbr$, 
$\lceil x\rceil$ (and $\lfloor x \rfloor$)
denote(s) the smallest integer not less than $x$  
(and the largest integer not exceeding $x$, respectively).  
The symmetric conformal prediction interval is then defined as 
\[
    \hat{\mathcal{C}}(X_{n+1})
    = \big[\hat f_1(X_{n+1}) - \hat{q},\;
           \hat f_1(X_{n+1}) + \hat{q} \big].
\]
Exchangeability yields the finite-sample guarantee
\[
 1-\alpha
 \leq \Pr\Big(Y_{n+1}\in \hat{\mathcal{C}}(X_{n+1})\Big) \leq
 1-\alpha + \frac{1}{n_2+1}.
\]
In symmetric conformal prediction, the interval length is
\[
    \La= 2 \hat{q}
\]
and its asymptotic behavior depends on different factors like the distribution of $\epsilon,$ and the choice of $(n_1,n_2)$, which we 
will investigate in the later sections.

%%%%%%%%%%%%%
\subsection{Asymmetric conformal prediction}
Symmetric intervals may be inefficient for skewed error distributions. Asymmetric split conformal prediction allows distinct upper and lower miscoverage levels~\cite{cordier2023flexible}. 
%[change to Linusson 2014 signed error]. 
Let $\alpha_L,\alpha_U\ge 0$ satisfy $\alpha_L+\alpha_U=\alpha \in (0,1)$.  
Define the signed conformity scores
\[
    \tilde\epsilon_i = Y_i - \hat f_1(X_i) , \qquad i \in \Dt.
\]
Let $\qt_l \equiv \tilde{q}_{n_2,\alpha_L}$ and $\qt_u \equiv \tilde{q}_{n_2,1-\alpha_U}$ denote the $\lfloor (n_2+1)\alpha_L\rfloor$-quantile of $\{\tilde\epsilon_i : i\in \Dt\}$ and the $\lceil (n_2+1)(1-\alpha_U)\rceil$-quantile of $\{\tilde\epsilon_i : i\in \Dt\}$ respectively.  
Then the asymmetric interval is
\[
    \tilde{\mathcal{C}}(X_{n+1})
    = \big[\hat f_1(X_{n+1}) + \qt_l,\;
           \hat f_1(X_{n+1}) + \qt_u \big].
\]
Exchangeability again ensures that
\[
  1-\alpha \leq   \Pr\Big(Y_{n+1}\in \tilde{\mathcal{C}}(X_{n+1})\Big) \leq  1-\alpha + \frac{2}{n_2+1} ,
\]
and the interval length is
\[
    \Lt
    = \qt_u - \qt_l.
\]
Later we will see that in many cases, the use of asymmetric intervals can be advantageous over the symmetric ones.

{%\it Revisit! 
Under the model $Y = f(X) + \epsilon,$ with $X \sim \mathbf{P}_X$ and $\epsilon\overset{\mathrm{iid}}{\sim} F_\epsilon,$ both symmetric and asymmetric split conformal intervals provide exact finite-sample coverage bounds, which primarily depend on the calibration sample size $n_2$ through the general upper bound $\frac{k}{n_2+1}$(with $k=1$ for the symmetric case and $k=2$ for the asymmetric case). On the other hand,  the length of the conformal interval is a random quantity based on both the training set size $n_1$ (through the estimator  $\hat{f}_{1}$) as well as the calibration set size $n_2$ (through the sample quantiles based on conformal scores). As a result, once the optimal split for the interval length is determined, the coverage accuracy/error admits a very definite nonasymptotic bound without further work. 
%we propse to determine the optimal 
%split size for the overall 
%best performance of the split conformal prediction  interval by balancing its 
%\emph{expected length} and its 
%\emph{variance}.
%of the random length 
%of the conformal prediction  interval. 
Therefore, in our analysis, we mainly focus on determining 
the optimal lengths of the prediction intervals
by  explicitly characterizing the asymptotic behaviors of 
%conformal intervals through 
their \emph{expected length} and the \emph{variance}. 
%with both quantities depending on the accuracy of the point predictor $\hat %f_1,$, the calibration quantiles, and the underlying noise distribution. 
The relative contributions of these terms determine how the sample should be split into the training and calibration sets. By developing suitable expansions of $\E(\La)$ (or $\E(\Lt)$) and $\Var(\La)$ (or $\Var(\Lt)$) for general $(n_1,n_2)$, in the next section we derive the optimal split 
proportion that minimizes the interval length while preserving valid coverage. 
%These results are developed in detail in the subsequent sections.
}

\section{Main results}\label{section:main-results}
%Optimal Data Splitting for Split Conformal Prediction}\label{section:optimal-spliting}
%%In this section, we present our main theoretical results.
%on optimal sample allocation in split conformal prediction. 
We begin by introducing some notation that will be used throughout the paper, followed by a description of the assumptions required for our analysis.  
Section 3.3 presents results on the  expectation and variance 
of the lengths of symmetric and asymmetric conformal prediction intervals. In Section 3.4, we specialize the results to three regression settings, namely linear regression, nonparametric regression, and neural network regression, 
and derive specific optimal split proportions in each case. Finally, in Section 
3.5, we describe a data-driven procedure for selecting the optimal split based on a subsampling technique.

\subsection{Notation}
For a  function $g$ on the real line,  we denote its first and second derivatives by $g'$ and $g''$, respectively. The inverse of the function $g$ is denoted by $g^{-1}$. Throughout the paper, $C_1, C_2,\dots$ denote constants that are independent of the sample size $n$. For two positive sequences $a_n$ and $b_n,$ we write $a_n=O(b_n)$ if there exists a constant $C$ such that $a_n \leq Cb_n$
for all $n$. Similarly, $a_n = o(b_n)$ indicates that $a_n/b_n \to 0,$ as $n\to\infty$. Furthermore, by $a_n \asymp b_n$,  we denote that $a_n=O(b_n)$ and $b_n=O(a_n),$ simultaneously. Recall that the floor and ceiling functions, denoted by $\lfloor x \rfloor$ and $\lceil x \rceil,$ represent the greatest integer less than or equal to $x,$ and the smallest integer greater than or equal to $x,$ respectively.
% For any function $g,$ denote $g',g'',\dots$ or $g^{(1)}, g^{(2)},\dots$ as the first derivative, second derivative, and so on. We denote $g^{-1}$ as the inverse function of the function $g$. By $C_1, C_2,\dots$ we denote constants independent of the sample size $n$. For two positive sequences $a_n$ and $b_n,$ $a_n=O(b_n)$ means there exist a constant $C$ such that $a_n \leq Cb_n$ and $a_n = o(b_n)$ means $a_n/b_n \to 0,$ as $n\to \infty.$ Furthermore, by $a_n \asymp b_n$ we denote that $a_n=O(b_n)$ and $b_n=O(a_n)$ simultaneously. We denote $\lfloor x \rfloor$ and $\lceil x \rceil$ as the greatest integer not greater than $x$ and smallest integer not smaller than $x,$ respectively. 

%In the regression case, 
For an estimator $\hat{f}_n$ of $f$ based on sample size $n$, define the estimation error at a point $x$ by $\hat{\gamma}_n(x) = \hat{f}_n(x) - f(x),$ which measures the discrepancy between the estimator $\hat{f}_n(x)$ and the true regression function $f(x)$. For notational conveniences, we simply write the powers  $(\hat{\gamma}_n(x))^k$ as $\hat{\gamma}_n^k(x)$ for any $k\geq2$. We denote by $\E_1,$ $\Var_1,$ and $\Cov_1$ the conditional expectation, variance, and covariance, respectively, given the proper training set $\Do$. For example, the $L_2$ error can be written as
\[ \E\int |\hat{f}_n(x) - f(x)|^2 \mathbf{P}_X(dx)=\E\E_1(\hat{\ga}_n^2(X)).\]
Similarly, the conditional variance is written as
\[ \E\Var_1(\hat{\ga}_n(X)) = \E\E_1(\hat{\ga}_n^2(X)) - \E\{\E_1(\hat{\ga}_n(X))\}^2.\]
Furthermore, for notational brevity, we will write $\hat{\ga}_{n_1} = \hat{f}_{n_1} - f$ as $\hg$.

% We denote $\E_1,$ $\Var_1,$ and $\Cov_1$ as the expectation, variance, and covariance, respectively, conditional on the proper training set $\Do.$ For example, we will write the $L_2$ error as
% \[ \E\int |\hat{f}_n(x) - f(x)|^2 \mathbf{P}_X(dx)=\E\E_1(\hat{\ga}_n^2(X)).\]

% \[ \E\Var_1(\hat{\ga}_n(X)) = \E\E_1(\hat{\ga}_n^2(X)) - \E\{\E_1(\hat{\ga}_n(X))\}^2\]

\subsection{Conditions}
For the symmetric split conformal prediction interval, define
\[m_\alpha=\left\lceil (n_2+1)(1-\alpha)\right\rceil,\qquad
\mu_\alpha=\frac{m_\alpha}{n_2+1},\qquad
\sigma_\alpha^2=\frac{\mu_\alpha(1-\mu_\alpha)}{n_2+2}.\]

\noindent For the asymmetric split conformal prediction interval, define
\[m_1=\left\lfloor (n_2+1)\alpha_L\right\rfloor,\qquad
m_2=\left\lceil (n_2+1)(1-\alpha_U)\right\rceil,\]
and
\[\mu_i=\frac{m_i}{n_2+1},\qquad
\sigma_i^2=\frac{\mu_i(1-\mu_i)}{n_2+2},
\qquad i=1,2,\]
together with
\[\rho_{12}=\frac{\mu_1(1-\mu_2)}{n_2+2}.\]

Let $S$ denote the population conformity score, with distribution function $F_S$. Throughout the paper,
\[S=
\begin{cases}
|\epsilon|, & \text{for symmetric split conformal prediction},\\
\epsilon, & \text{for asymmetric split conformal prediction}.
\end{cases}\]
Furthermore, let $\delta\in(0,1]$ be a constant independent of $n$. We make the following assumptions.
\begin{enumerate}[label={(A.\arabic*)}]
    \item\label{A3} For some integer $K\geq 1,$ suppose there exists a positive, strictly increasing sequence of constants $\{\beta_k\}_{k=1}^K,$ independent of $n,$ such that 
    \[\E\int|\hat{f}_{n}(x) - f(x)|^k \;\mathbf{P}_X(dx) \asymp n^{-\beta_k}.\]
    \item \hspace{-0.2cm}$_r$\label{A1} Suppose there exists two sequences defined by $a_{1n} = \delta^{-1}n^{-\gamma},$ for some $\gamma\in(0,(\beta_k-\beta_r)/k),k>r$ and $a_{2n} = \delta^{-1}\sqrt{\log(n)/n},$ such that for each relevant quantile level $\mu$,
    \[\delta<\inf_{|y|\leq a_{1n},|x|\leq a_{2n}} F_S'\Big(\fs^{-1} (\mu+x)+y\Big) \leq \sup_{|y|\leq a_{1n},|x|\leq a_{2n}} F_S'\Big(\fs^{-1} (\mu+x)+y \Big) <\delta^{-1}, \] 
    \[\sup_{|y|\leq a_{1n},|x|\leq a_{2n}} \Big|F_S''\Big(\fs^{-1}(\mu+x)+y\Big)\Big|<\delta^{-1},\]
    % Furthermore, let $a_{1n} = \delta^{-1}n^{-\beta_l},$ for some $l<(\beta_k-\beta_2)/k,\;k\geq4$ and $a_{2n} = \delta^{-1}\sqrt{\log(n)/n},$ then
    and
    \[ \sup_{|y|\leq a_{1n},|x|\leq a_{2n}} \Big|F_S''\Big(\fs^{-1}(\mu+x)+y\Big)-F_S''\Big(\fs^{-1}(\mu)\Big) \Big| = o(1), \text{ as } n\to\infty. \]
    For Theorem~\ref{thm:1}, corresponding to the symmetric interval, this condition is imposed with $F_S=F_{|\epsilon|}$ and $\mu=\mu_\alpha$. For Theorem~\ref{thm:2}, corresponding to the asymmetric interval, it is imposed with $F_S=F_\epsilon$ and $\mu\in\{\mu_1,\mu_2\}$.
    \item\label{A2} There exists a constant $\beta_\epsilon>0,$ independent of $n,$ such that for all $x\geq \delta^{-1},$ 
    \[ F_\epsilon(-x) + (1-F_\epsilon(x)) \leq \delta^{-1} |x|^{-\beta_\epsilon}.\] 
    %\[|F_\epsilon^{-1}(u)| \leq \delta^{-1}\Big\{u^{-\beta_\epsilon}+ (1-u)^{-\beta_\epsilon}\Big\}.\]
\end{enumerate}

Assumption \ref{A3} imposes polynomial moment bounds on the estimation error of the regression estimator $\hat f_n$. Specifically, it requires that for some integer $K \ge 1,$ the $L_k$-risk $\E\!\int |\hat f_n(x)-f(x)|^k\,\mathbf{P}_X(dx)$ decays at the rate $n^{-\beta_k}$ for $k=1,\dots,K$. Such conditions are standard in nonparametric regression and statistical learning theory, and are satisfied by a broad range of estimators, including kernel smoothers, local polynomial estimators, and certain neural network regressors; see, for example, \cite{gyorfi2002distribution, fan2018local}. In the context of conformal inference, moment bounds of this form ensure an accurate approximation of the residual distribution, and have appeared in recent theoretical analyses of predictive inference~\cite{lei2018distribution}.
% Assumption \ref{A3} imposes polynomial moment bounds on the estimation error of the regression estimator $\hat f_n$. Specifically, it requires that for some integer $K \ge 1,$ the $L_k$-risk $\E\!\int |\hat f_n(x)-f(x)|^k\,P_X(dx)$ decays at a rate $n^{-\beta_k}$ for $k=1,\dots,K$. Such conditions are standard in nonparametric regression and learning theory and hold for a wide range of estimators, including kernel smoothers, local polynomial estimators, and certain neural network regressors; see, for example, Györfi et al.\ (2002) and Fan and Gijbels (1996). In the context of conformal inference, moment bounds of this type ensure that the estimated residual distribution is well approximated and appear in recent theoretical studies of predictive inference (Lei et al., 2018; Barber et al., 2021).

Assumption \ref{A1}$_r$ imposes smoothness and local stability conditions on the noise distribution function $F_S$ (where $S$ is either $|\epsilon|$ or $\epsilon$). Specifically, it requires that the first and second derivatives of $F_S$ remain uniformly bounded away from zero and infinity within a shrinking neighborhood. These conditions ensure that the quantile function does not become excessively flat or steep. Such conditions are classical in asymptotic quantile theory and in deriving Bahadur-type expansions~\cite{serfling2009approximation}. The final requirement in Assumption~\ref{A1}$_r,$ which controls the difference between second derivatives evaluated at nearby points, is essentially a local Lipschitz condition on $F_S''$. Such derivative-boundedness and local smoothness conditions are standard in the analysis of quantile estimators and residual-based inference; see, for example, \cite{van2000asymptotic}, where smoothness of the noise density plays a key role in obtaining uniform expansions of sample quantiles.
% Assumption \ref{A1}$_r$ imposes smoothness and local stability conditions on the noise distribution function $F_\epsilon$. Specifically, it requires that the first and second derivatives of $F_\epsilon$ remain uniformly bounded away from zero and infinity within a shrinking neighborhood. These conditions ensure that the quantile function does not become excessively flat or steep. Such conditions are classical in asymptotic quantile theory and in deriving Bahadur-type expansions (Serfling, 1980; van der Vaart, 1998). The final requirement in Assumption~\ref{A1}$_r,$ which controls the difference between second derivatives evaluated at nearby points, is essentially a local Lipschitz condition on $\fe''$. Such derivative-boundedness and local smoothness conditions are standard in the analysis of quantile estimators and residual-based inference; see, for example, Bhattacharya and Ranga Rao (1976, Chapter 8) and Bickel (1967), where smoothness of the density is required for uniform expansions of sample quantiles.

Assumption \ref{A2} imposes a polynomial tail condition on the noise distribution $F_\epsilon$. This assumption allows for moderately heavy-tailed noise and is substantially weaker than sub-Gaussian or sub-exponential assumptions. Polynomial tail bounds of this form ensure that the contribution of extreme residuals is asymptotically negligible, thereby facilitating higher-order asymptotic approximations of quantile-based statistics. Similar tail conditions are common in the asymptotic theory of order statistics; see, for example, \cite{bickel1967some} and \cite{stigler1969}. In the present setting, Assumption~\ref{A2} is crucial for guaranteeing that contributions from extreme residuals are asymptotically negligible in the expansion of the conformal quantile, so that the approximation error is dominated by the smooth interior behavior described in Assumption~\ref{A1}$_r$.
% Assumption \ref{A2} imposes a polynomial tail condition on the noise distribution $F_\epsilon$. This assumption allows for moderately heavy-tailed noise and is substantially weaker than sub-Gaussian or sub-exponential requirements. Polynomial tail bounds of this form guarantee the existence of suitable moments of the residuals and ensure that empirical quantile estimators possess stable variance and negligible higher-order approximation errors; see, for example, Petrov (1975) and Hall and Welch (1985), who analyze quantile behavior under polynomial-decay tails. In the present setting, Assumption~\ref{A2} is needed to ensure that contributions from extreme residuals are asymptotically negligible in the expansion of the conformal quantile, so that the approximation error is dominated by the smooth interior behavior captured in Assumption~\ref{A1}$_r$.

Taken together, assumptions \ref{A3}--\ref{A2} define a flexible and realistic framework under which the regression estimator converges at a polynomial rate, the residual quantile function is locally smooth, and the noise distribution exhibits only moderate polynomial tail decay. These conditions are mild and broadly consistent with regularity assumptions commonly used in nonparametric regression and quantile inference. They ensure that the analytic approximation of the conformal prediction interval length holds uniformly and that higher-order remainder terms remain negligible.

\subsection{Main results}
In this subsection, we present the main theoretical results of the paper. Under Assumptions \ref{A3}--\ref{A2}, we derive expressions for the expectation and variance of the length of the split conformal prediction interval and use these results to characterize the optimal data-splitting strategy. 

The following theorem provides the mean and variance of the length of the symmetric split conformal prediction interval.

\begin{theorem}\label{thm:1}
Let $\beta_\epsilon > 2/n_2$ and $\alpha\in \Big[(\beta_\epsilon+1)/(\beta_\epsilon(n_2+1)),1-1/(\beta_\epsilon(n_2+1))\Big)$.
\begin{itemize} 
     \item[(a)] Suppose Assumptions \ref{A3}, \ref{A1}$_1$ and \ref{A2} hold, then we have,
    \beqs
\begin{split}
    \E\big(\La\big) = 2\fme^{-1}(\mu_\alpha) +  2H_{|\epsilon|}\Big(\fme^{-1}(\mu_\alpha)\Big) \E\E_1(\hg(X)) + \sigma_\alpha^2 {\fme^{-1}}''(\mu_\alpha) + o\bigg( \frac{1}{n_1^{\beta_1}}+\frac{1}{n_2}\bigg).
\end{split}
\eeqs    
Furthermore, under Assumption \ref{A1}$_2$,
\beqs
\begin{split}
    \Var\big(\La\big) =  4\Big\{ H_{|\epsilon|}\Big(\fme^{-1}(\mu_\alpha)\Big) \Big\}^2 \Var\E_1(\hg(X)) + 4\sigma_\alpha^2 \Big{\{\fme^{-1}}'(\mu_\alpha)\Big\}^2 + o\bigg(\frac{1}{n_1^{\beta_2}}+\frac{1}{n_2}\bigg),
\end{split}
\eeqs
where $H_{|\epsilon|}(a) = -\{\fe'(a) - \fe'(-a)\}/\fme'(a).$

\item[(b)] If the distribution of $\epsilon$ is symmetric and Assumption \ref{A1}$_2$ holds, then
    \beqs
\begin{split}
    \E\big(\La\big) = 2\fme^{-1}(\mu_\alpha) +  2G_{\epsilon}\Big(\fme^{-1}(\mu_\alpha)\Big) \E\E_1(\hg^2(X)) + \sigma_\alpha^2 {\fme^{-1}}''(\mu_\alpha) + o\bigg( \frac{1}{n_1^{\beta_2}}+\frac{1}{n_2}\bigg).
    \end{split}
\eeqs    
Furthermore, under Assumption \ref{A1}$_4$,
    \beqs
\begin{split}
    \Var\big(\La\big) =  4\Big\{ G_{\epsilon}\Big(\fme^{-1}(\mu_\alpha)\Big) \Big\}^2 \Var\E_1(\hg^2(X)) + 4\sigma_\alpha^2 \Big{\{\fme^{-1}}'(\mu_\alpha)\Big\}^2 + o\bigg(\frac{1}{n_1^{\beta_4}}+\frac{1}{n_2}\bigg),
\end{split}
\eeqs
where $G_{\epsilon}(a) = -\fe''(a)/(2\fe'(a)).$

%     \item[(a)]Under the assumptions \ref{A3}, \ref{A1}$_2$ and \ref{A2}, we have,
%     \beqs
% \begin{split}
%     \E\big(\La\big) = 2\fme^{-1}(\mu_\alpha) +  2H_{|\epsilon|}\Big(\fme^{-1}(\mu_\alpha)\Big) \E\E_1(\hg(X)) + \sigma_\alpha^2 {\fme^{-1}}''(\mu_\alpha) + o\bigg( \frac{1}{n_1^{\beta_1}}+\frac{1}{n_2}\bigg), \\
%     \Var\big(\La\big) =  4\Big\{ H_{|\epsilon|}\Big(\fme^{-1}(\mu_\alpha)\Big) \Big\}^2 \Var\E_1(\hg(X)) + 4\sigma_\alpha^2 \Big{\{\fme^{-1}}'(\mu_\alpha)\Big\}^2 + o\bigg(\frac{1}{n_1^{\beta_2}}+\frac{1}{n_2}\bigg),
% \end{split}
% \eeqs
% where $H_{|\epsilon|}(a) = -\{\fe'(a) - \fe'(-a)\}/\fme'(a).$

% \item[(b)] Suppose the distribution of $\epsilon$ is absolutely symmetric. Then under the assumptions \ref{A3}, \ref{A1}$_2$ and \ref{A2}, we have,
%     \beqs
% \begin{split}
%     \E\big(\La\big) = 2\fme^{-1}(\mu_\alpha) +  2G_{\epsilon}\Big(\fme^{-1}(\mu_\alpha)\Big) \E\E_1(\hg^2(X)) + \sigma_\alpha^2 {\fme^{-1}}''(\mu_\alpha) + o\bigg( \frac{1}{n_1^{\beta_2}}+\frac{1}{n_2}\bigg), \\
%     \Var\big(\La\big) =  4\Big\{ G_{\epsilon}\Big(\fme^{-1}(\mu_\alpha)\Big) \Big\}^2 \Var\E_1(\hg^2(X)) + 4\sigma_\alpha^2 \Big{\{\fme^{-1}}'(\mu_\alpha)\Big\}^2 + o\bigg(\frac{1}{n_1^{\beta_4}}+\frac{1}{n_2}\bigg),
% \end{split}
% \eeqs
% where $G_{\epsilon}(a) = -\fe''(a)/(2\fe'(a)).$
\end{itemize}
\end{theorem}

\begin{remark}
The leading term
$2F^{-1}_{|\epsilon|}(\mu_\alpha)$
in the expansion of $\E(\hat L)$ corresponds to the expected length obtained when the
regression function is known exactly. Since
\[
\mu_\alpha=\frac{\lceil(n_2+1)(1-\alpha)\rceil}{n_2+1}
\ge 1-\alpha,
\]
and $F^{-1}_{|\epsilon|}$ is an increasing function, this quantity is always at least as
large as the oracle interval length
$
2\fme^{-1}(1-\alpha).$
Thus, even in the absence of estimation error, the split conformal prediction interval
exhibits a positive finite-sample bias arising from the calibration quantile. The remaining
terms in Theorem~\ref{thm:1} quantify the additional increase in interval length due to
estimating the regression function. Consequently, minimizing these higher-order terms
yields the shortest attainable split conformal prediction interval while preserving the
finite-sample coverage guarantee.
\end{remark}

\begin{remark}
In the case of a symmetric noise distribution, the leading term $2F^{-1}_{|\epsilon|}(\mu_\alpha)$ depends only on the calibration sample size through the empirical quantile level \(\mu_\alpha\) and is independent of the proper training sample size. Since its dependence on \(n_2\) contributes only to the constant of the calibration error of order \(n_2^{-1}\), it does not affect the asymptotically optimal order of the training--calibration split. Consequently, the optimal allocation of samples is determined by balancing the regression estimation error,
\[
\E\E_1(\hat\gamma_1^2(X))
\quad\text{and}\quad
\Var\E_1(\hat\gamma_1^2(X)),
\]
which decrease as the proper training sample size \(n_1\) increases, against the overall calibration error of order $n_2^{-1},$ which decreases as the calibration sample size \(n_2\) increases.

For example, suppose
\[
\E\E_1(\hat\gamma_1^2(X))
\asymp
n_1^{-\beta},
\qquad
\Var\E_1(\hat\gamma_1^2(X))
\asymp
n_1^{-2\beta},
\]
for some \(\beta>0\). Then a bias-optimal split satisfies
\[
n_1^{-\beta}\asymp n_2^{-1},
\]
whereas an MSE-optimal split satisfies
\[
n_1^{-2\beta}\asymp n_2^{-1}.
\]
The corresponding values of \(\beta\) for several commonly used regression models are derived in Section~\ref{subsec: examples}.
\end{remark}

\begin{remark}
The quantities $\E\E_1(\hat\gamma_1(X))$ and
$\Var\E_1(\hat\gamma_1(X))$ appearing in
Theorem~\ref{thm:1}(a) may converge faster than the upper bounds
$n_1^{-\beta_1}$ and $n_1^{-\beta_2}$ implied by Assumption~\ref{A3}. Thus,
the stated rates should be interpreted as worst-case controls rather than exact
orders. Furthermore, the remainder term can be made arbitrarily small by imposing
higher-order smoothness assumptions on the error distribution. In particular, if
Assumption~\ref{A1}$_r$ holds with $r=n_0$ for some positive integer $n_0$, then
the remainder becomes
\(
o(n_1^{-\beta_{n_0}}+n_2^{-1}),
\)
showing that increasingly accurate asymptotic expansions can be obtained by assuming
additional differentiability of the error distribution.
\end{remark}

The next theorem gives the mean and variance of the length of the asymmetric split conformal prediction interval.

\begin{theorem}\label{thm:2}
Suppose $\beta_\epsilon\geq 2/(n_2-1),$ $\alpha_L \in ((\beta_\epsilon+1)/(\beta_\epsilon(n+1)), 1/2]$, and $\alpha_U \in [(\beta_\epsilon+1)/(\beta_\epsilon(n+1)),1/2),$ with $\alpha_L+\alpha_U=\alpha$. Then, under Assumptions \ref{A3}, \ref{A1}$_2$ and \ref{A2}, we have
    \beqs
\begin{split}
    \E\big(\Lt\big)& = \Big\{\fe^{-1}(\mu_{2}) - \fe^{-1}(\mu_{1})\Big\} \\&\hspace{1cm} + \Big\{ G_\epsilon\Big(\fe^{-1}(\mu_{2})\Big) - G_\epsilon\Big(\fe^{-1}(\mu_{1})\Big)\Big\} \E\Var_1\Big(\hg(X)\Big) \\& \hspace{2cm} + \Big\{\frac{\sigma_2^2}{2}\;{\fe^{-1}}''(\mu_{2}) - \frac{\sigma_1^2}{2}\;{\fe^{-1}}''(\mu_{1}) \Big\} + o\bigg(\frac{1}{n_1^{\beta_2}}+\frac{1}{n_2}\bigg).    
\end{split}
\eeqs
Furthermore, if Assumption \ref{A1}$_4$ holds, then
    \beqs
\begin{split}
    \Var\big(\Lt\big) & = \Big\{G_\epsilon\Big(\fe^{-1}(\mu_{2})\Big) - G_\epsilon\Big(\fe^{-1}(\mu_{1})\Big) \Big\}^2 \Var\Var_1(\hg(X)) \\&\hspace{1cm} + \sigma_{2}^2\;\Big\{{\fe^{-1}}'(\mu_{2})\Big\}^2 + \sigma_{1}^2\;\Big\{{\fe^{-1}}'(\mu_{1})\Big\}^2 \\&\hspace{2cm} - 2\rho_{12}\; {\fe^{-1}}'(\mu_{2}) {\fe^{-1}}'(\mu_{1}) + o\bigg(\frac{1}{n_1^{\beta_4}}+\frac{1}{n_2}\bigg),
\end{split}
\eeqs
where $G_\epsilon(a) = -\fe''(a)/(2\fe'(a)).$
\end{theorem}

\begin{remark}
Theorem~\ref{thm:2} shows that the approximation error for the expected length of the split conformal prediction interval can be substantially smaller for asymmetric intervals than for the classical symmetric interval. The key distinction is that the leading term in Theorem~\ref{thm:2} depends on
\(
\E\Var_1(\hat\gamma_1(X)),
\)
whereas the corresponding term in Theorem~\ref{thm:1}(b) depends on
\(
\E\E_1(\hat\gamma_1^2(X)).
\)
Since
\[
\Var_1(\hat\gamma_1(X))
=
\E_1(\hat\gamma_1^2(X))
-
\{\E_1(\hat\gamma_1(X))\}^2,
\]
we always have
\[
\E\Var_1(\hat\gamma_1(X))
\le
\E\E_1(\hat\gamma_1^2(X)),
\]
with strict inequality whenever the conditional bias of the regression estimator is nonzero. Consequently, the leading approximation error for asymmetric conformal prediction may be of strictly smaller order than that of the symmetric interval. Moreover, unlike the symmetric case, this improvement does not require the noise distribution to be symmetric, making asymmetric conformal prediction theoretically advantageous under considerably weaker assumptions. Similar conclusions apply to the variance expansion.
\end{remark}

\subsection{Examples} \label{subsec: examples}
In this subsection, we derive the expressions that appeared in the main theoretical results, which enable practical determination of the optimal split ratio between the proper training and calibration samples. For illustration, we focus on a symmetric conformal prediction interval with a symmetric error distribution. We consider three regression settings: linear regression, nonparametric regression via the Nadaraya-Watson kernel estimator, and fully connected neural network regression. These examples illustrate how the general theory applies across different regression paradigms. Since $\sigma_\alpha^2 \asymp n_2^{-1}$ always holds, it suffices to characterize $\E\E_1(\hg^2(X))$ and $\Var\E_1(\hg^2(X)),$ equivalently the expectation and variance of the $L_2$ errors, as functions of the (proper) training size $n_1$ for each model.

\subsubsection{Linear regression}\label{subsec:linear_regression_theory}
We begin by considering a random effects model with the linear regression function $f(x) = x^\top\beta,$ where $x\in\R^d$ and $\beta\in\R^d$ are the covariates and regression coefficients, respectively. For illustrative purposes, let $\{(X_i,Y_i)\}_{i=1}^{n_1}$ denote a collection of $n_1$ independent and identically distributed samples. We consider a Gaussian design with $X_i\overset{iid}{\sim} \mathcal{N}(0, I_d),$ where $I_d$ is the $d$-dimensional identity matrix, and errors $\epsilon_i \overset{iid}{\sim} \mathcal{N}(0,1)$. Let $X$ denote the design matrix and $Y$ the response vector, then the ordinary least squares linear regression estimate of the regression function is given by $\hat{f}(x) = x^\top\hat{\beta},$ where $\hat{\beta} = (X^\top X)^{-1}X^\top Y$ and we have
% First, we consider the random effects model with the linear regression function $f(x) = x^\top\beta,\;x_i\in\R^d,$ where $\beta\in\R^d$ are the regression coefficients. For illustrative purposes, given $n_1$ iid samples $\{(X_i,Y_i)\}_{i=1}^{n_1},$ we consider a Gaussian design where $X_i\overset{iid}{\sim} \mathcal{N}(0, I_p),$ where $I_p$ is the $d$-dimensional identity matrix and $\epsilon_i \overset{iid}{\sim} \mathcal{N}(0,1)$. Let $X$ denote the design matrix and $Y$ denote the response vector, then the ordinary least squares linear regression estimate $\hat{f}(x) = x^\top\hat{\beta},$ where $\hat{\beta} = (X^\top X)^{-1}X^\top Y$ and we have
\[ \E\E_1(\hg^2(X)) = \frac{d}{n_1-d-1} \asymp \frac{1}{n_1},\]
and
\[ \Var\E_1(\hg^2(X)) = 2\E\Big[\text{tr}(X^\top X)^{-2}\Big] + \Var\Big[\text{tr}(X^\top X)^{-1}\Big] \asymp \frac{1}{n_1^2}, \]
provided $n_1 > d+3$. 
% Then, in the ordinary least squares linear regression, the sizes of the training set and the calibration set should satisfy $n_1\asymp n_2$ for a bias-optimal split. That is, the size of the training set should be approximately equal to the size of the calibration set. For an MSE-optimal split, $n_1\asymp \sqrt{n_2}$ should satisfy.
Then, in ordinary least squares linear regression, a bias-optimal split requires $n_1 \asymp n_2$, that is, the training and calibration sample sizes are of the same order. For an MSE-optimal split, one requires $n_1 \asymp \sqrt{n_2}$.

% Then, in the ordinary least squares linear regression, following the main results, we should use $n_1\asymp n_2,$ i.e., the size of the proper training set needs to be almost equal in size to the calibration set.

\subsubsection{Nonparametric regression}
We consider a nonparametric regression framework in which the goal is to estimate the regression function $f(x)=\E[Y \mid X=x] $ from $n_1$ iid samples $\{(X_i, Y_i)\}_{i=1}^{n_1}$. We assume that the covariates $X_i\in[-a, a]^d,$ for some $a>0$ are uniformly distributed, and the noise terms $\{\epsilon_i\}_{i=1}^{n_1}$ are iid with $\epsilon_i\overset{iid}{\sim}\mathcal{N}(0,1)$. For estimation, we employ the Nadaraya-Watson kernel estimator, a local-constant regression method that smooths the data using a bandwidth parameter $h>0$ and a kernel function $\mathcal{K}:\R^d \to \R_+$ that satisfies $\int\mathcal{K}(u)\,du=1$. The estimator at a point $x\in\R^d$ is defined as 
\[ \hat{f}_{h,n_1}(x)  = \frac{\sum_{i=1}^{n_1} \mathcal{K}\Big(\frac{x - X_i}{h}\Big) Y_i} {\sum_{i=1}^{n_1} \mathcal{K}\Big(\frac{x - X_i}{h}\Big)}, \]
which can be viewed as a locally weighted average of the observed responses $Y_i,$ with larger weights assigned to observations where the covariates $X_i$ are closer to $x$. For illustrative purposes, we assume that the kernel $\mathcal{K}$ is bounded, symmetric, and continuously differentiable with compact support. Finally, if we choose the bandwidth $h\asymp n_1^{-1/(4+d)},$ standard calculation shows
\[ \E\E_1(\hg^2(X)) \asymp \frac{1}{n_1^{4/(4+d)}},\]
and
\[ \Var\E_1(\hg^2(X)) \asymp \frac{1}{n_1^{8/(4+d)}}, \]
provided the regression function has bounded first and second derivatives. Then, in the Nadaraya-Watson kernel regression, the theoretically justified choice for the bias-optimal data splitting is $n_1\asymp n_2^{1+d/4},$ implying that the training sample should be considerably larger than the calibration sample and $n_1\asymp n_2^{1/2+d/8}$ for the MSE-optimal choice.

\subsubsection{Neural network regression}\label{section:nn_regression}
For the final example, we consider a fully connected neural network with the Rectified Linear Unit (ReLU) activation function. In this setting, exact expressions for the expectation and variance of the $L_2$-error are not yet available in the literature. However, bounds on the expected $L_2$-error have been derived in \cite{kohler21}, under certain structural assumptions on the regression function. To state these results, we rely on the following two definitions.

\begin{definition}[{$(p,C)$-smoothness}, \cite{kohler21}]
Let $p = q + s$ for some $q \in \mathbb{N}_0$ and $0 < s \le 1$.  
A function $m : \mathbb{R}^d \to \mathbb{R}$ is called $(p,C)$-smooth if, for every 
multi-index $\alpha = (\alpha_1,\dots,\alpha_d) \in \mathbb{N}_0^d$ with 
$\sum_{j=1}^d \alpha_j = q,$ the partial derivative
\[
\dfrac{\partial^q m}{\partial x_1^{\alpha_1}\cdots \partial x_d^{\alpha_d}}
\]
exists and satisfies
\[
\biggl|
\frac{\partial^q m}{\partial x_1^{\alpha_1}\cdots \partial x_d^{\alpha_d}}(x) - 
\frac{\partial^q m}{\partial x_1^{\alpha_1}\cdots \partial x_d^{\alpha_d}}(z)
\biggr|
\le C \|x - z\|^s,
\]
for all $x,z \in \mathbb{R}^d,$ where $\|\cdot\|$ denotes the Euclidean norm.
\end{definition}

With the definition of $(p,C)$-smoothness, we define the class of hierarchical composition models as follows.

\begin{definition}[Hierarchical composition models, \cite{kohler21}]
For $l = 1$ and smoothness constraint $\mathcal{P} \subseteq (0,\infty) \times \mathbb{N},$ the space of hierarchical composition models is defined as
\begin{align*}
    \mathcal{H}(1, \mathcal{P})  := &
\bigg\{
h : \mathbb{R}^d \to \mathbb{R} : 
h(a) = m \big( a_{\pi(1)}, \ldots, a_{\pi(K)} \big),
\ \text{where } m : \mathbb{R}^K \to \mathbb{R} \text{ is }\\&\;\; (p, C)\text{-smooth for some } (p,K) \in \mathcal{P}
\text{ and } \pi : \{1, \ldots, K\} \to \{1, \ldots, d\}
\bigg\}.
\end{align*}
For $l > 1,$ we recursively define
\begin{align*}
\mathcal{H}(l, \mathcal{P}) := &
\bigg\{
h : \mathbb{R}^d \to \mathbb{R} : 
h(x) = m \big( f_1(a), \ldots, f_K(a) \big),
\ \text{where } m : \mathbb{R}^K \to \mathbb{R} \text{ is } \\&\;\; (p, C)\text{-smooth for some } (p,K) \in \mathcal{P}
\text{ and } f_i \in \mathcal{H}(l-1, \mathcal{P})
\bigg\}.
\end{align*}
\end{definition}
Let the corresponding regression function $f$ belong to the class $\mathcal{H}(l,p)$ for some $l\in\mathbb{N}$ and $\mathcal{P}\subseteq (0,\infty)\times\mathbb{N}$. Then according to Theorem~1 in \cite{kohler21}, and under some regularity conditions, we have
\[\E\E_1(\hg^2(X)) = O\Big( (\log(n))^6  \max_{(p,K)\in \mathcal{P}} n^{-\frac{2p}{2p+K}} \Big).\]
Although the exact order is not available, for a bias-optimal split, we can use the approximate relation $n_1 \asymp n_2^{1+\bar{K}/2\bar{p}},$ where $(\bar{p},\bar{K})\in \mathcal{P}$ is chosen such that $(\bar{p},\bar{K})=\arg\min_{(p,K)\in \mathcal{P}} p/K$. This implies that the relative size of the proper training set to the calibration set depends on the smoothness and order constraint $\mathcal{P}$.

\subsection{Data based optimal choice of training ratio}
When the sample size $n$ is very large, it may be difficult to determine the optimal split, as doing so requires repeatedly training the model under different splits, which may be computationally infeasible in practice. However, using our results, one can employ a data-driven procedure to select the optimal split via subsampling (cf. \cite{hall1995blocking}).

Based on the preceding results, we obtain that the optimal split satisfies $n_1 \asymp n_2^b$ for some known constant $b>0$. Let $\hat{c}_\ell$ denote the empirical estimate of the ratio of the training sample size to the total sample size, computed from a subsample of size $\ell (<n)$. The corresponding estimate $\hat{c}_n$ for the full sample size $n$ is then determined as the solution to the equation
% From the above results, we have $n_1 \asymp n_2^\ell,$ for some known $\ell>0.$ Suppose $\hat{c}_m$ be the empirical choice of the ratio of the training sample to the total sample based on $m (<n)$ samples. Then $\hat{c}_n$ would be the solution to the equation
\begin{equation} \label{eq:ratio}
    x = \frac{n^{b-1}\hat{c}_\ell}{\ell^{b-1}(1-\hat{c}_\ell)^b} (1-x)^b.
\end{equation}

The following algorithm provides a way to approximate the ratio $\hat{c}_{\ell}$ where $\ell<n$.

\bigskip
\noindent\textbf{Algorithm Description.}
Given $\alpha \in [1/(n_2+1),1),$ and $\ell(<n),$ we choose $\hat{c}_\ell$ as follows:
\begin{enumerate}
    \item Divide the dataset randomly into $K$ disjoint subsets $\{\mathcal{Q}_{k}\}_{k=1}^K,$ each containing $\ell$ samples, such that $K\ell\leq n$. %\red{[(possibly discarding the remaining $n - K\ell$ samples--it is okay if we do not mention)]}
   % \item For each $k=1,\dots,K,$ further randomly partition the subset $\mathcal{Q}_{m,k}$ into a training subset $\mathcal{Q}_{m,k}^{(1)}$ of size $m_1$ and a test subset $\mathcal{Q}_{m,k}^{(2)}$ of size $m_2=m-m_1$.
    \item Choose a grid $\G_g = \{ p_1, \ldots, p_g \} \in (0,1)^g$ of size $g$ (e.g., $\G_6 = \{0.4, 0.5, \ldots, 0.9\}$).
    \item For each $k = 1, 2, \ldots, K$:
    \begin{enumerate}
         \item For each candidate proportion $p_i \in \mathcal{G}_g,$ $i=1,\dots,g,$ further split the subsets $\mathcal{Q}_{k}$ into a proper training set $\mathcal{T}_{k,i}$ of size $\lfloor \ell p_i \rfloor$ and a calibration set $\mathcal{C}_{k,i}$ of size $\ell - \lfloor \ell p_i \rfloor$.
        %Subsequently, compute the empirical coverage probability $\widehat{\mathcal{P}}(\mathcal{Q}_{m,k}^{(2)}, p_i)$ on the same test subset $\mathcal{Q}_{m,k}^{(2)}$.
        \item  Using $\mathcal{T}_{k,i}$ and $\mathcal{C}_{k,i},$ compute prediction interval length $\hat{\mathcal{L}}(\mathcal{T}_{k,i},\mathcal{C}_{k,i},p_i)$ and select the target proportion $\hat{p}_{k}= \arg\min _{p_i \in \G_g}\hat{\mathcal{L}}(\mathcal{T}_{k,i},\mathcal{C}_{k,i},p_i).$        %\begin{align*}
        %\G_g^* = \bigg\{p_i \in \G_g : \widehat{\mathcal{P}}(\mathcal{Q}_{m,k}^{(2)},p_i) \in \bigg(1-\alpha-\delta, 1-\alpha+\frac{1}{\lceil m_1(1-p_i) \rceil+1}+\delta \bigg)\bigg\}.
        %\end{align*}
    \end{enumerate}
    %\item Let $\hat{p}_k,$ $k=1,\ldots,K$ be the optimal ratios.
    \item Set $\hat{c}_\ell = \frac{1}{K}\sum_{k=1}^{K}\hat{p}_k$
    \item (Optional refinement) If greater precision is required, repeat the Steps 2--5 using a finer grid centered around the current estimate $\hat{c}_\ell,$ and update $\hat{c}_\ell$ accordingly.
\end{enumerate}
\begin{remark}
     In practice, one can choose $\ell$ such that $K = 5$ or $10$ or more based on the sample sizes and model computation time. In this way, we need to train the model only on a dataset of size $\ell p,$ for $p\in \G$, for which we can have $\ell p <<  n_1$. However, it is important to note that, for the algorithm to function correctly, at least one value in $\G_g$ must be less than $p^*$ such that $\alpha>1/(\lceil \ell(1-p^*) \rceil +1)$. %Therefore, \red{[incomplete sentence]}
\end{remark}
\begin{remark}
    To the best of our knowledge, existing empirical practice has largely treated the training-calibration split in split conformal prediction as fixed rather than tunable. A common convention in the literature is to use equal-sized proper training and calibration sets, as adopted in empirical evaluations of split conformal methods with ridge regression, random forests, and neural networks in~\cite{romano2019conformalized}. This convention effectively promotes a 1:1 split as a default choice, despite the lack of theoretical or empirical justification for its optimality in terms of interval length. 
    Related recent work has explicitly targeted the problem of prediction set length, but from a different perspective. In particular, recent work formulates length optimization as a constrained problem that minimizes expected prediction set size subject to coverage constraints, and develops a minimax-based procedure that learns covariate-dependent thresholds to construct shorter sets~\cite{kiyani2024length}. Their framework is implemented on top of diverse predictive models, including linear models and deep neural network architectures such as ResNet and large language model backbones. However, across these experiments, the allocation between training and calibration data is not treated as a parameter to be optimized and is fixed within each setup, irrespective of the underlying model. 
    In contrast, our results indicate that a 1:1 split is not generally optimal. The allocation of samples between training and calibration interacts with the underlying learning algorithm, as different models induce different trade-offs between estimation accuracy and calibration precision. Consequently, the split ratio should be regarded as an algorithm-dependent design parameter, and selecting it adaptively can yield shorter prediction intervals while preserving nominal coverage.
\end{remark}

\section{Experimental results}\label{section:experimetns}
In this section, we evaluate our theoretical results using both synthetic and real-world datasets across three regression settings: linear regression, nonparametric regression, and neural networks. In synthetic datasets, each experiment involves $n+1$ samples, consisting of $n$ samples used for training and calibration, and one sample for testing. For each underlying model, we vary the training sample size $n$ and randomly partition the data into training and calibration sets. To ensure reliability, all experiments are repeated 1,000 times for each training-calibration split ratio. We employ split conformal prediction to construct prediction intervals with a nominal coverage level of $1-\alpha=0.9$ and report the average prediction interval length.

% In this section, evaluate our theoretical results on both synthetic and real-world datasets in three regression settings: linear regression, nonparametric regression, and neural network regression. For each experiment, we consider $n+1$ samples, including $n$ training and calibration samples and $1$ test sample. For each setting, we vary the training sample size $n$ and randomly split the data into a training set and a calibration set. To ensure reliability, we repeat all experiments 1,000 times for each training-calibration split ratio. We apply split conformal prediction to construct prediction intervals at nominal coverage $1-\alpha=0.90$ and report the average prediction interval length.

\subsection{Linear regression} \label{sim: linreg} 
We examine the following linear regression model:
\begin{align}\label{eq:linear_regression}
    y_i = x_i^\top\beta + \epsilon_i, \qquad i=1,\ldots,n,
\end{align}
where $y_i\in\R$ is the response variable, $x_i\in[0,1]^5$ denotes the regressors, $\beta\in\R^5$ represents the regression coefficients, and $\epsilon_i\in\R$ is the error term. 
Using equation \eqref{eq:linear_regression}, we generate synthetic datasets with varying sample sizes $\{100,300,500,800,2000\}$. The regressors $x_i$ are sampled independently from a uniform distribution on $[0,1]^5$. We consider three types of error distributions:
\begin{itemize}
\item Normal distribution:
\begin{align*}
    \epsilon_i \sim \mathcal{N}(0, 1)
\end{align*}
\item Student's t-distribution with 5 degrees of freedom (variance-normalized):
\begin{align*}
    \epsilon_i \sim \frac{t(5)}{\sqrt{5/3}}
\end{align*}
\item Lognormal distribution (standardized to zero mean and unit variance):
\begin{align*}
    \epsilon_i \sim \frac{\text{Lognormal}(0, 1) - \exp(0.5)}{\sqrt{(\exp(1) - 1)\exp(1)}}
\end{align*}
\end{itemize}
We repeat the experiments 1,000 times for each error setting and sample size and report the average prediction interval lengths in the first row of Figure~\ref{fig:comparison}.
As shown in Section~\ref{subsec:linear_regression_theory}, for linear regression models a bias-optimal split requires the training and calibration set sizes to satisfy $n_1 \asymp n_2$. That is, the training set size should be approximately equal to the calibration set size. Equivalently, this corresponds to setting $b=1$ in Equation~\ref{eq:ratio}. Solving the equation for $b=1$ yields $x=\hat{c}_{\ell}$, which is consistent with the numerical results in the first row of Figure~\ref{fig:comparison}. As the total sample size increases, the optimal prediction interval length continues to occur when the training and calibration sets are approximately equal in size.
% \blue{We repeat the experiments for each error setting and sample size 1,000 times and report the average prediction interval lengths in the first row of Figure~\ref{fig:comparison}. As anticipated, for linear regression problems, the minimum prediction interval length is achieved when the sizes of the training and calibration sets are almost equal.
% As we have shown in section~\ref{subsec:linear_regression_theory}, for the linear regression models, the sizes of the training set and the calibration set should satisfy $n_1\asymp n_2$ for a bias-optimal split. That is, the size of the training set should be approximately equal to the size of the calibration set. In other words, the parameter $b=1$ in equation~\ref{eq:ratio}. Solving this equation for $b=1$, we obtain $x=\hat{c}_{\ell}$ which is confirmed by our numerical results in the first row of Figure~\ref{fig:comparison}. As we increase the number of data points, the optimal prediction interval length is obtained when the number of training and calibration sets are almost equal.}

\begin{figure}[htbp]
    \centering
    \includegraphics[width=\linewidth]{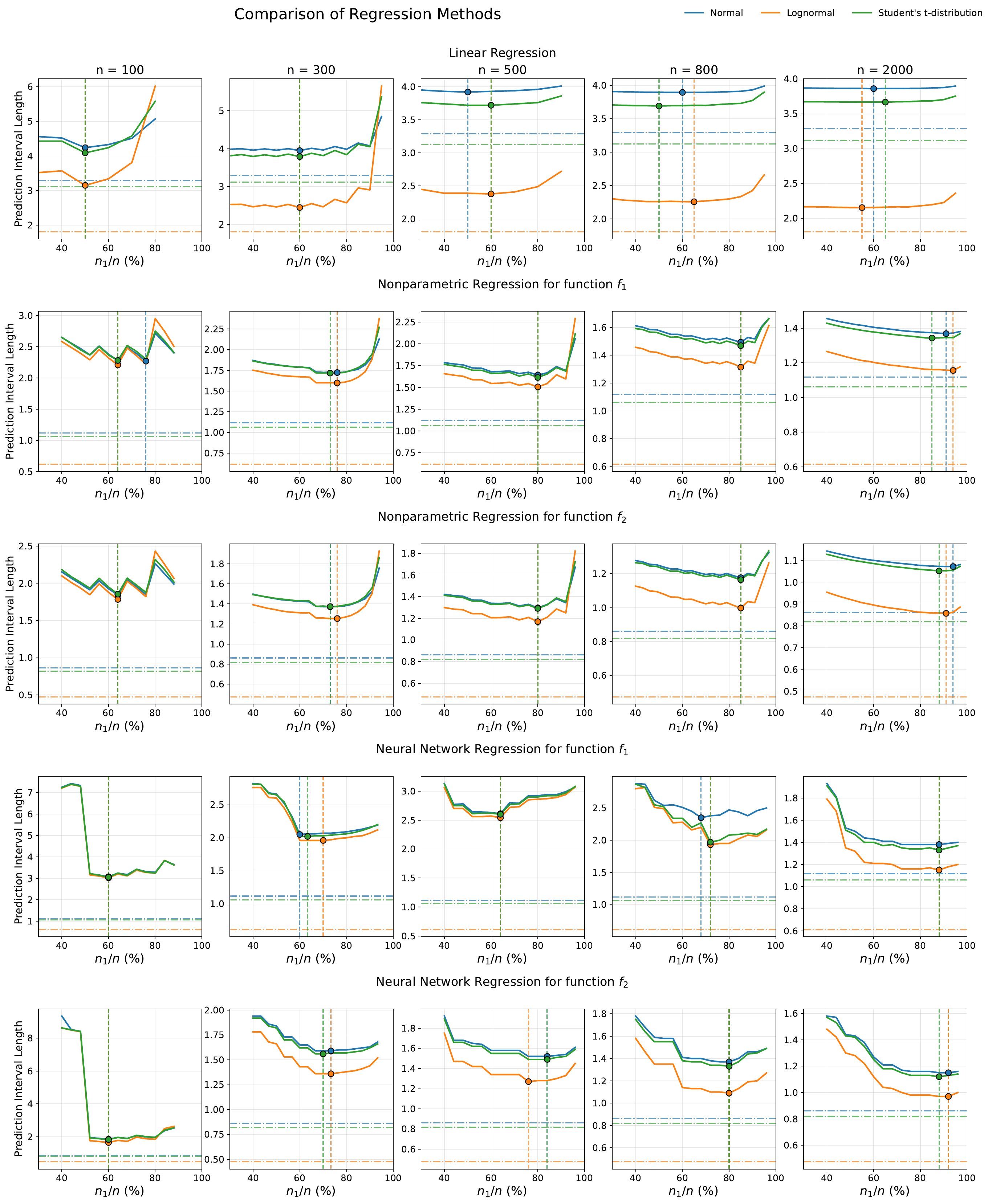}
    \caption{Prediction interval lengths for synthetic datasets as a function of the data split ratio. Solid lines represent the experimental results. Dashed lines indicate the locations of the minimum prediction interval length. Dash-dot lines correspond to the true prediction interval lengths (i.e., the oracle prediction lengths).}
    \label{fig:comparison}
\end{figure}

\subsection{Nonparametric regression} \label{sim: npreg}
In this section, we investigate two settings of a nonparametric regression model of the form 
\begin{align*}
y_{i,j} = f_j(x_i) + \lambda_j\epsilon_i, \quad i = 1, \ldots, n, \quad j=1,2, 
\end{align*}
where $y_{i,j}\in\R$ is the response variable, $f_j$ is the regression function, $x_i=(x_i^{(1)},\ldots,x_i^{(5)})^\top\in\R^5$ is a five-dimensional predictor which is uniformly distributed over the region $[0,1]^5$ independently, $\epsilon_i$ is the independent error term, which is also independent of $x_i$ as defined in Subsection~\ref{sim: linreg}, and $\lambda_1,\lambda_2\ge 0$ are fixed scaling parameters. For reasons of comparability, the scaling parameters $\lambda_j$ are chosen to match $20\%$ of the typical range of variation of $f_j(x)$ under the distribution of $x$. This range is chosen via the interquartile range (IQR) of $f_j(x)$ computed from many simulated draws of $x$, so that the noise level is scaled relative to the signal magnitude. These scaling parameters are set to $\lambda_1=0.340$ and $\lambda_2=0.262$.
 The regression functions under consideration are: 
\begin{itemize}
\item Function $f_1$:
\begin{align*}
    f_1(x) = 3x^{(1)}+\tan (x^{(2)})+(x^{(3)})^3+\log(x^{(4)}+0.1)+\sqrt{x^{(5)}+0.1}
\end{align*}
\item Function $f_2$:
\begin{align*}
    f_2(x) = \exp(\|x\|),
\end{align*}
where $x=(x^{(1)},\ldots,x^{(5)})^\top$.
\end{itemize}
Following the experimental setup in the linear regression setting, we generate synthetic datasets with sample sizes $\{100, 200, 500, 800, 2000\}$. For each sample size, we consider three types of error distributions: normal, lognormal, and Student's t-distribution. Each experiment is repeated 1,000 times to ensure statistical reliability. The average prediction interval lengths computed via split conformal prediction for the two functions are shown in the second and third rows of Figure~\ref{fig:comparison}. 
% As we were expecting, to obtain the minimum prediction interval length, the training sample should be considerably larger than the calibration sample.
From our theoretical analysis of Nadaraya-Watson kernel regression, the bias-optimal data split satisfies $n_1 \asymp n_2^{1+d/4}$, where $d$ denotes the covariate dimension. In our experiments, $d=5$, which yields $n_1 \asymp n_2^{9/4}$. This scaling indicates that nonparametric regression requires a larger training set than calibration set, since accurate nonparametric estimation relies primarily on having more training data.
Under this relationship, the corresponding ratio parameter in Equation~\ref{eq:ratio} is $b=9/4$. Substituting this value into Equation~\ref{eq:ratio} and considering the setting $n=2000$ and $\ell=300$, together with the empirically observed minimum prediction interval length at $\hat{c}_{\ell}\approx 0.75$, yields $x\approx 0.91$. This theoretical value aligns closely with the experimental findings reported in the last columns of the second and third rows of Figure~\ref{fig:comparison}. Overall, these results provide empirical support for the proposed algorithm for selecting the split ratio and validate the practical relevance of Equation~\ref{eq:ratio} in the nonparametric regression setting.
%\ref{fig:comparison}. 

\subsection{Neural network regression}
In this section, we evaluate our theoretical results on neural network regression using synthetic and real-world datasets. 

\subsubsection{Synthetic dataset}
For the synthetic dataset, we again consider the setting in Subsection~\ref{sim: npreg}: 
\begin{align*}%\label{eq:nonparametric_regression} 
y_{i,j} = f_j(x_i) + \lambda_j\epsilon_i, \qquad i = 1, \ldots, n, \quad j=1,2.
\end{align*}
We generate synthetic datasets for sample sizes $\{100, 300, 500, 800, 1000\}$ and consider three types of error distributions: normal, lognormal, and Student's t-distribution. Each configuration is evaluated over 1,000 repetitions. We report the average prediction interval lengths for each setup, with the results displayed in the fourth and fifth rows of Figure \ref{fig:comparison}. In our simulations, the neural network consists of one hidden layer for sample sizes of $\{100, 300\}$ and two hidden layers for sample sizes of $\{500, 800, 2000\}$. The network is trained for a maximum of 100 epochs in all configurations.
From our theoretical analysis of neural networks, a bias-optimal split approximately satisfies $n_1 \asymp n_2^{1+\bar{K}/2\bar{p}}$. In our setting, $\bar{p}=1$ and $\bar{K}=5$ for the functions $f_1$ and $f_2$, which gives $n_1 \asymp n_2^{7/2}$. As introduced in Section~\ref{section:nn_regression}, $(\bar{p},\bar{K})$ denote the smoothness and order parameters in $\mathcal{P}$ selected as the minimizer of $p/K$. Similar to the nonparametric regression case, this scaling indicates that neural networks also require a larger training set than a calibration set.
Under this relationship, the corresponding ratio parameter in Equation~\ref{eq:ratio} is $b=7/2$. Substituting this value into Equation~\ref{eq:ratio} with $n=2000$ and $\ell=100$, and using the empirically observed minimum prediction interval length at $\hat{c}_{\ell}\approx 0.6$, yields $x\approx 0.93$ (similar results are obtained if we consider higher values of $\ell$). This theoretical value closely matches the experimental results reported in the last columns of the fourth and fifth rows of Figure~\ref{fig:comparison}. These findings further support the proposed algorithm for selecting the split ratio.
% As expected, the smallest prediction intervals are obtained when the number of training samples greatly exceeds the number of calibration samples. This aligns with the understanding that larger training sets allow the neural network to better capture underlying patterns in the data, whereas smaller calibration sets ensure the conformal prediction intervals remain accurate. 

\subsubsection{Concrete compressive strength dataset}
To evaluate the empirical performance of the proposed algorithm on a benchmark dataset, we consider the Concrete Compressive Strength dataset~\cite{UCIConcrete2019}. The neural network used in our experiments has two hidden layers and is trained for up to 100 epochs. The dataset is partitioned into training and test sets using an 80:20 split, yielding 824 training and 206 test observations. The training portion is further divided into training and calibration sets for conformal prediction.
To apply Equation~\ref{eq:ratio} and the proposed procedure for selecting a split ratio that minimizes prediction interval length, we set $\bar{p}=1$ and $\bar{K}=8$ (see Section~\ref{section:nn_regression}). Although these parameters can be chosen in a data-dependent way, we only choose them heuristically. Under these choices, $n_1 \asymp n_2^5$, corresponding to $b=5$. To implement Equation~\ref{eq:ratio}, the training data are partitioned into four subsets, giving $\ell=206$ in Equation~\ref{eq:ratio}. Experiments are conducted over a range of split ratios within each subset, and split conformal prediction is applied to the test data. Results are averaged across the four subsets, as shown in the left panel of Figure~\ref{fig:NN_concrete}.
The estimate $\hat{c}_{\ell}=0.66$ is obtained. Substituting these values into Equation~\ref{eq:ratio} yields $x \approx 0.88$. This theoretical value is consistent with the empirical results shown in the right panel of Figure~\ref{fig:NN_concrete}, which are based on the full training dataset. This agreement provides further support for the proposed split-ratio selection procedure.

\begin{figure}[!h]
    \centering
    \includegraphics[width=\linewidth]{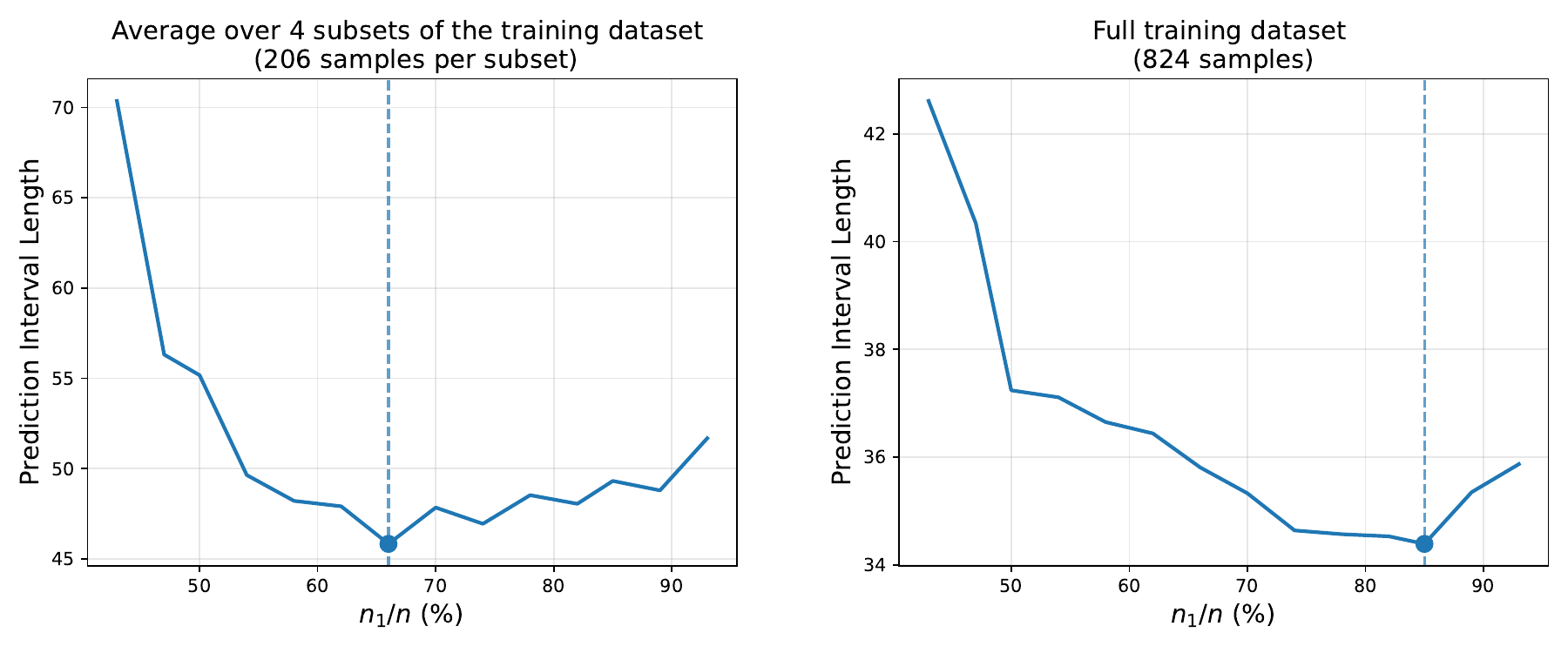}
    \caption{Prediction interval length for different training samples on the concrete compressive strength dataset.}
    \label{fig:NN_concrete}
\end{figure}

\section{Discussion and future directions}\label{section:discussion}
While our experiments are conducted on datasets of moderate size, where one could, in principle, evaluate multiple train-calibration splits to search for shorter prediction intervals, such an approach is often impractical in real-world deployments. In many safety-critical and high-stakes domains, such as healthcare, autonomous systems, and financial risk management, datasets often contain a large number of samples, making model training computationally demanding, time-consuming, and subject to operational or regulatory constraints. Repeated retraining across numerous candidate split ratios is therefore inefficient and, in some cases, infeasible. Additionally, in certain applications, data acquisition itself is costly or limited, further reducing the viability of brute-force split selection. Our method instead provides a principled approach to selecting the data split with the objective of minimizing prediction interval length without requiring repeated retraining on a large dataset. To the best of our knowledge, existing work that aims to reduce conformal prediction interval length typically relies on random or heuristic data splits rather than explicitly optimizing the train-calibration partition. This highlights the practical relevance and novelty of our approach.

\section{Proof of the main results}
First, we present the proof of Theorem~\ref{thm:2} and then move on to prove Theorem~\ref{thm:1}. 
\begin{proof}[Proof of Theorem~\ref{thm:2}]
Fix  $\alpha_2 \in \Big[(\beta_\epsilon+1)/\beta_\epsilon(n_2+1)),1-1/\beta_\epsilon(n_2+1))\Big)$.
Let $\hf$ be the conditional CDF of $\tilde{\epsilon}_i = \epsilon_i-\hat{\gamma}(X_i),$ $i\in\Dt$,  given $\Do$.
Then, 
\begin{align*}
    \E(\qt_u) &= \E ( \E(\qt_u\mid \Do)),\\
    \Var(\qt_u) &= \Var (\E(\qt_u\mid \Do)) + \E (\Var(\qt_u\mid \Do)),
\end{align*}
where $\qt_u$ is the $m_2=\lceil(n_2+1)(1-\alpha_2)\rceil$-th quantile of $\{\tilde{\epsilon}_i\}_{i\in\Dt}$ (same argument follows for $m_1=\lfloor(n_2+1)\alpha_1\rfloor$-th order statistic  of $\{\tilde{\epsilon}_i\}_{i\in\Dt}$).
 Let  $U_1,\ldots,U_{n_2}$ be a set of uniformly distributed random variables on $(0,1)$. Define $U_{(m_2)}$ as the $m_2^{th}$ order statistic of $U_1,\ldots,U_{n_2}$, with mean  
  $\mu_2$ and variance $\sigma_2^2$. 
Recall that, we have $a_{2n_2} = C\sqrt{\log(n_2)/n_2},$ for some constant $C>0$. Then, using Taylor's theorem, we have
\beq \label{eq: exp}
    \begin{split}
         &\E(\qt_u\mid\Do) \\& = \E_1\hf^{-1}(U_{(m_2)}) \\
    & = \E_1\hf^{-1}(U_{(m_2)}) \mathbbm{1}(|U_{(m_2)}-\mu_2|\leq a_{2n_2}) \\&\hspace{3cm}+ \E_1\hf^{-1}(U_{(m_2)}) \mathbbm{1}(|U_{(m_2)}-\mu_2|>a_{2n_2}) \\
    & = \E_1\bigg[ \hf^{-1}(\mu_2) + (U_{(m_2)}-\mu_2){\hf^{-1'}}(\mu_2) \\ &\hspace{3cm}+ \frac{1}{2} (U_{(m_2)}-\mu_2)^2 {\hf^{-1''}}(\mu_2) \bigg] \mathbbm{1}(|U_{(m_2)}-\mu_2|\leq a_{2n_2}) \\
    & \hspace{1cm} + \frac{1}{2} \E_1 (U_{(m_2)}-\mu_2)^2 \left( \hf^{-1''}(\mu_*) - \hf^{-1''}(\mu_2) \right) \mathbbm{1}(|U_{(m_2)}-\mu_2|\leq a_{2n_2}) \\
    & \hspace{1cm} + \E_1\hf^{-1}(U_{(m_2)}) \mathbbm{1}(|U_{(m_2)}-\mu_2|>a_{2n_2}) \\
    & = \hf^{-1}(\mu_2) + 0 + \frac{\sigma_2^2}{2} \hf^{-1''}(\mu_2) + R_{1n},   
    \end{split}
\eeq

where $\mu_*$ is some point between $U_{(m_2)}$ and $\mu_2$ and $R_{1n}$ is the remainder term defined as
\begin{equation*}
    \begin{split}
    R_{1n} & = -\E_1\bigg[ \hf^{-1}(\mu_2) + (U_{(m_2)}-\mu_2){\hf^{-1'}}(\mu_2) \\&\hspace{4cm}+ \frac{1}{2} (U_{(m_2)}-\mu_2)^2 {\hf^{-1''}}(\mu_2) \bigg] \mathbbm{1}(|U_{(m_2)}-\mu_2|>a_{2n_2}) \\
    & \hspace{1cm} + \frac{1}{2} \E_1(U_{(m_2)}-\mu_2)^2 \left( \hf^{-1''}(\mu_*) - \hf^{-1''}(\mu_2) \right) \mathbbm{1}(|U_{(m_2)}-\mu_2|\leq a_{2n_2}) \\
    & \hspace{1cm} + \E_1\hf^{-1}(U_{(m_2)}) \mathbbm{1}(|U_{(m_2)}-\mu_2|>a_{2n_2}).
    \end{split}
\end{equation*} 
Then, we obtain
\beq \label{eq: r1n}
    \begin{split}
    |R_{1n}| %& \leq \hf^{-1}(\mu_\alpha) \Prob(|U_{(m)}-\mu_\alpha|>a_{2n_2}) \\
    % &  \qquad + |\hf^{-1'}(\mu_\alpha)\E|U_{(m)}-\mu_\alpha|\mathbbm{1}(|U_{(m)}-\mu_\alpha|>a_{2n_2}) \\
    % & \qquad + \frac{1}{2}|\hf^{-1''}(\mu_\alpha)|\E|U_{(m)}-\mu_\alpha|^2\mathbbm{1}(|U_{(m)}-\mu_\alpha|>a_{2n_2}) \\
    % & \qquad + \frac{1}{2}\E|U_{(m)}-\mu_\alpha|^2|\hf^{-1''}(\mu_\alpha+(\mu_*-\mu_\alpha)) - \hf^{-1''}(\mu_\alpha)| \mathbbm{1}(|U_{(m)}-\mu_\alpha|\leq a_{2n_2}) \\
    % & \qquad + \E|\hf^{-1}(U_{(m)})| \mathbbm{1}(|U_{(m)}-\mu_\alpha|> a_{2n_2}) \\
    & \leq |\hf^{-1}(\mu_2) |\Prob(|U_{(m_2)}-\mu_2|>a_{2n_2}) \\
    & \qquad + |\hf^{-1'}(\mu_2)|\; \E|U_{(m_2)}-\mu_2|\mathbbm{1}(|U_{(m_2)}-\mu_2|>a_{2n_2}) \\
    & \qquad + %\frac{1}{2}
    |\hf^{-1''}(\mu_2)|\;\E|U_{(m_2)}-\mu_2|^2\mathbbm{1}(|U_{(m_2)}-\mu_2|>a_{2n_2}) \\
    & \qquad + %\frac{\sigma_{\alpha_2}^2}{2}
    \sigma_2^2\;\underset{|x|\leq a_{2n_2}}{\sup}|\hf^{-1''}(\mu_2+x)-\hf^{-1''}(\mu_2)| \\
    & \qquad + \E_1|\hf^{-1}(U_{(m_2)})| \mathbbm{1}(|U_{(m_2)}-\mu_2|> a_{2n_2}).
    \end{split}
\eeq
A standard Hoeffding bound for Beta random variables gives,
\begin{equation*}
    \mathbf{P}(|U_{(m_2)}-\mu_2|> a_{2n_2}) \leq 2\exp(-2n_2a_{2n_2}^2) = o\bigg(\frac{1}{n_2}\bigg),
\end{equation*}
and later we show that,
\begin{equation} \label{eq: eifn}
    \hf^{-1}(\mu_2) = \fe^{-1}(\mu_2) + R_{2n},
\end{equation}
with $E|R_{2n}|=o(1).$ Therefore, by assumptions \ref{A3}, \ref{A1}$_2$ and \ref{A2}, we get that $\E|R_{1n}| = o(1/n_2).$ Next, we show that the quantiles (conditional on $\Do$) of $\tilde\epsilon_i$ are close to the quantiles of $\epsilon_i.$

First, note that for some fixed $a$ and $h,$ by Taylor's theorem
\beq \label{eq: hf}
\begin{split}
    \hf(a+h) & = \E_1\fe (a+\hg(X)+h)\\ &= \fe(a) + \fe'(a)\Big\{\E_1(\hg(X))+h \Big\} \\& \hspace{2cm} + \frac{1}{2}\fe''(a)\Big\{\E_1(\hg^2(X))+2h\E_1(\hg(X))+h^2\Big\} + R_{3n}(a,h),
\end{split}
\eeq
where $R_{3n}(a,h)$ is a remainder term given as
\begin{equation*}
    \begin{split}
    R_{3n}(a,h) & = -\E_1\bigg[ \fe(a) + \Big\{\hg(X)+h\Big\}\fe'(a) \\&\hspace{4cm}+ \frac{1}{2} \Big\{\hg(X)+h\Big\}^2\fe''(a)\bigg] \mathbbm{1}\Big(|\hg(X)+h|> a_{1n_1}\Big)\\
    & \hspace{1cm} + \frac{1}{2} \left( \fe''\Big(a+\theta^*\Big(\hg(X)+h\Big)\Big) - \fe''(a) \right) \\&\hspace{5.4cm}\E_1 \Big\{\hg(X)+h\Big\}^2  \mathbbm{1}\Big(|\hg(X)+h|\leq a_{1n_1}\Big) \\
    & \hspace{1cm} + \E_1\fe \Big(a+\hg(X)+h\Big) \mathbbm{1}\Big(|\hg(X)+h|> a_{1n_1}\Big).
    \end{split}
\end{equation*}
for some $\theta^*\in (0,1)$ and $ a_{1n_1}$ is as defined in \ref{A1}$_2$. Therefore, using H\"older inequality
\begin{equation} \label{eq: r2n}
    \begin{split}
    & |R_{3n}(a,h)|\\ &\leq 2 \mathbf{P}\Big(|\hg(X)+h|> a_{1n_1}\mid \Do\Big) \\&\hspace{0.5cm} + \Big\{ \mathbf{P}\Big(|\hg(X)+h|> a_{1n_1}\mid \Do\Big) \Big\}^{1/2}\bigg[ \fe'(a)\Big\{\E_1\Big(\hg(X)+h\Big)^2\Big\}^{1/2} 
    \\ & \hspace{7.5cm} + \fe''(a) \Big\{\E_1\Big(\hg(X)+h\Big)^4\Big\}^{1/2} \bigg] \\
    & \hspace{2cm} + \sup_{|y|\leq a_{1n_1}} \left| \fe''(a+y) - \fe''(a) \right|  \E_1 \Big\{\hg(X)+h\Big\}^2.
    \end{split}
\end{equation} 
Note that, by Markov's inequality and using assumption \ref{A3}, we can get
\begin{equation} \label{eq: prob_r2n}
    \begin{split}
    \E \mathbf{P}(|\hg(X)-\E_1(\hg(X))|> a_{1n_1}\mid \Do) \leq \frac{C\E\E_1\Big(|\hg(X)|^k\Big)}{ a_{1n_1}^k} = o\bigg(\frac{1}{n_1^{\beta_2}}\bigg).
    \end{split}
\end{equation} 
Then, by choosing $h = -\E_1(\hg(X)) + G_{\epsilon} (a) \Var_1(\hg(X)) ,$ where $G_\epsilon(a) = -\fe''(a)/(2\fe'(a))$ and $a = \fe^{-1}(\mu_2)$ and  we obtain,
\beq \label{eq: ifn}
\begin{split}
    \hf^{-1}(\mu_2) = \fe^{-1}(\mu_2) -\E_1(\hg(X)) + G_\epsilon\Big(\fe^{-1}(\mu_2)\Big) \Var_1(\hg(X)) + R_{4n},
\end{split}
\eeq 
where using \eqref{eq: r2n}, \eqref{eq: prob_r2n} and assumptions \ref{A3}, we get $\E|R_{4n}|= o(1/n_1^{\beta_2}).$ This shows that \eqref{eq: eifn} follows from \eqref{eq: ifn}. Therefore, we have
\beq \label{key1}
\begin{split}
    \E(\qt_u) &= \E\E(\qt_u\mid\Do) \\&= \fe^{-1}(\mu_2) -\E\E_1(\hg(X)) + G_\epsilon\Big(\fe^{-1}(\mu_2)\Big) \E\Var_1(\hg(X)) \\&\hspace{6cm} + \frac{\sigma_2^2}{2} {\fe^{-1}}''(\mu_2)  + o\bigg(\frac{1}{n_1^{\beta_2}}+\frac{1}{n_2}\bigg).
\end{split}
\eeq
Also, by using \ref{A3} and \ref{A1}$_4,$ from \eqref{eq: r2n} and \eqref{eq: prob_r2n}, we get $\E|R_{4n}|=o(n_1^{-\beta_4})$ and consequently we get
\beq
\begin{split} \label{eq: vare}
    \Var\E(\qt_u\mid\Do) &= \Var\E_1(\hg(X)) +  \Big\{G_\epsilon\Big(\fe^{-1}(\mu_2)\Big)\Big\}^2 \Var\Var_1(\hg(X)) \\& \hspace{0.5cm} - 2G_\epsilon\Big(\fe^{-1}(\mu_2)\Big) \Cov\Big( \E_1(\hg(X)), \Var_1(\hg(X)) \Big) + o\bigg(\frac{1}{n_1^{\beta_4}}  +\frac{1}{n_2}\bigg).
\end{split}
\eeq
Next, we calculate $\E\Var(\qt_u\mid\Do)$. Using Taylor's theorem, we have
\beq \label{eq: cvar}
    \begin{split}
    &\Var(\qt_u\mid\Do)\\ 
    & = \Var_1\hf^{-1}(U_{(m_2)}) \\
    & = \Var_1\bigg[ \bigg\{ \hf^{-1}(\mu_2) + (U_{(m_2)}-\mu_2){\hf^{-1'}}(\mu_2) \bigg\} \\&\hspace{1.5cm} - \bigg\{ \hf^{-1}(\mu_2) + (U_{(m_2)}-\mu_2){\hf^{-1'}}(\mu_2) \bigg\} \mathbbm{1}(|U_{(m_2)}-\mu_2|> a_{2n_2}) \\& \hspace{1.5cm} + (U_{(m_2)}-\mu_2)({\hf^{-1'}}(\mu_*)-{\hf^{-1'}}(\mu_2)) \mathbbm{1}(|U_{(m_2)}-\mu_2| \leq a_{2n_2}) \bigg] \\
    & \hspace{1cm} + \Var_1\hf^{-1}(U_{(m_2)}) \mathbbm{1}(|U_{(m_2)}-\mu_2|>a_{2n_2}) \\
    & \hspace{1cm} - \E_1\bigg[\hf^{-1}(U_{(m_2)}) \mathbbm{1}(|U_{(m_2)}-\mu_2|\leq a_{2n_2})\bigg] \\&\hspace{5cm} \E_1\bigg[\hf^{-1}(U_{(m_2)}) \mathbbm{1}(|U_{(m_2)}-\mu_2|>a_{2n_2})\bigg] \\
    & = \sigma_2^2 \Big\{ \hf^{-1'}(\mu_2)\Big\}^2 + R_{5n},   
    \end{split}
\eeq
where $R_{5n}$ is the remainder term such that, following similar calculations as in \eqref{eq: r1n}, under assumptions \ref{A1}$_4$ and \ref{A2}, $\E|R_{5n}|=o(1/n_2).$
Therefore,
\beq \label{eq: var}
\begin{split}
    \E\Var(\qt_u\mid\Do)
    &= \E\Var_1(\hf^{-1}(\mu_2))\\
    &= \sigma_2^2 \E\Big\{ \hf^{-1'}(\mu_2)\Big\}^2 + o\bigg(\frac{1}{n_2}\bigg)\\
    &= \sigma_2^2 \Big\{{\fe^{-1}}'(\mu_2)\Big\}^2 + o\bigg(\frac{1}{n_1^{\beta_4}}+\frac{1}{n_2}\bigg).
\end{split}
\eeq

Finally, we get 
\beq \label{key2}
\begin{split}
    \Var(\qt_u) & = \Var\E_1(\hg(X)) + \Big\{G_\epsilon\Big(\fe^{-1}(\mu_2)\Big)\Big\}^2 \Var\Var_1(\hg(X)) \\&\hspace{2cm} - 2G_\epsilon\Big(\fe^{-1}(\mu_2)\Big) \Cov\Big( \E_1(\hg(X)), \Var_1(\hg(X)) \Big) \\& \hspace{6cm} +\sigma_2^2 \Big{\{\fe^{-1}}'(\mu_2)\Big\}^2 + o\bigg(\frac{1}{n_1^{\beta_4}}+\frac{1}{n_2}\bigg).
\end{split}
\eeq

For the covariance, note that if $\alpha_1<(1-\alpha_2),$ then we have
\begin{equation*}
     \Cov(\qt_l,\qt_u) = \Cov \Big(\E(\qto\mid \Do), \E(\qtt\mid \Do)\Big) + \E \Cov(\qto, \qtt \mid \Do),
\end{equation*}
where $\qt_l$ and $\qtt$ are the $m_{1}=\lfloor(n_2+1)\alpha_1\rfloor$ and $m_2=\lceil(n_2+1)(1-\alpha_2)\rceil$-th quantiles of $\{\tilde{\epsilon}_i\}_{i\in\Dt},$ as previously defined and $U_{(m_{1})},$ $U_{(m_{2})}$ are the $m_{1}^{th}$ and $m_{2}^{th}$ order statistic of $U_1,\ldots,U_{n_2}\overset{iid}{\sim}Unif(0,1),$ with means $\mu_1=m_1/(n_2+1)$ and $\mu_2=m_2/(n_2+1),$ and covariance $\rho_{12} = m_1(n_2+1-m_2)/((n_2+1)^2(n_2+2)).$ Then, following \eqref{eq: exp} and \eqref{eq: ifn}, we have
\beqs
\begin{split}
    & \Cov \Big(\E(\qto\mid \Do), \E(\qtt\mid \Do)\Big)\\ &= \Var\E_1(\hg(X)) +  G_\epsilon\Big(\fe^{-1}(\mu_1)\Big) G_\epsilon\Big(\fe^{-1}(\mu_2)\Big) \Var\Var_1(\hg(X)) \\&\hspace{1cm}  - \bigg\{G_\epsilon\Big(\fe^{-1}(\mu_1)\Big) + G_\epsilon\Big(\fe^{-1}(\mu_2)\Big) \bigg\}\Cov\Big( \E_1(\hg(X)), \Var_1(\hg(X)) \Big) \\&\hspace{2cm} + o\bigg(\frac{1}{n_1^{\beta_4}} +\frac{1}{n_2}\bigg).
\end{split}
\eeqs 
Furthermore, following similar calculations as in \eqref{eq: cvar} and \eqref{eq: var}, it is straightforward to show that
\beqs
\begin{split}
    \E \Cov(\qto, \qtt \mid \Do)
    &= \E\Cov_1(\hf^{-1}(\mu_1),\hf^{-1}(\mu_2))\\
    &= \E\Big( \hf^{-1'}(\mu_1) \hf^{-1'}(\mu_2)\; \Cov_1(U_{(m_1)},U_{(m_2)}) \Big)  + o\bigg(\frac{1}{n_2}\bigg)\\
    &= \rho_{12}\; {\fe^{-1}}'(\mu_1) {\fe^{-1}}'(\mu_2) + o\bigg(\frac{1}{n_1^{\beta_4}}+\frac{1}{n_2}\bigg).
\end{split}
\eeqs
Then we get,
\beq \label{key3}
\begin{split}
    &  \Cov(\qto,\qtt) \\ &= \Var\E_1(\hg(X)) +  G_\epsilon\Big(\fe^{-1}(\mu_1)\Big) G_\epsilon\Big(\fe^{-1}(\mu_2)\Big) \Var\Var_1(\hg(X)) \\&\hspace{1cm}  - \bigg\{G_\epsilon\Big(\fe^{-1}(\mu_1)\Big) + G_\epsilon\Big(\fe^{-1}(\mu_2)\Big) \bigg\}\Cov\Big( \E_1(\hg(X)), \Var_1(\hg(X)) \Big) \\&\hspace{2cm} + \rho_{12} {\fe^{-1}}'(\mu_1) {\fe^{-1}}'(\mu_2) + o\bigg(\frac{1}{n_1^{\beta_4}} +\frac{1}{n_2}\bigg).
\end{split}
\eeq

Note that, $\Lt = \tqto - \tqtt,$ where $\tqto \equiv \qt_{n_2, 1-\alpha_U}$ and $\tqtt \equiv \qt_{n_2, \alpha_L}$ such that $\alpha_L<(1-\alpha_U).$ Therefore,
\beqs
\begin{split}
    \E(\Lt) &= \E(\tqto) - \E(\tqtt),\\
    \Var(\Lt) &= \Var(\tqto) + \Var(\tqtt) - 2\Cov(\tqto,\tqtt).
\end{split}
\eeqs
Then the result follows from \eqref{key1}, \eqref{key2} and \eqref{key3} by considering $\alpha_1=\alpha_L$ and $\alpha_2 = \alpha_U.$
\end{proof}

\begin{proof}[Proof of Theorem~\ref{thm:1}]
Proof of this theorem closely follows the proof of Theorem~\ref{thm:2}.
For $\alpha \in \Big[(\beta_\epsilon+1)/(\beta_\epsilon(n_2+1)),1-1/(\beta_\epsilon(n_2+1))\Big),$ we have 
\begin{align*}
    \E(\qa) &= \E\E(\qa\mid \Do),\\
    \Var(\qa) &= \Var \E(\qa\mid \Do) + \E \Var(\qa\mid \Do).
\end{align*}
Let $\tf$ be the CDF of $|\hat{\epsilon}_i| = |\epsilon_i-\hat{\gamma}(X_i)|,$ $i=n_1+1,\dots,n,$ given $\Do$. Assume $U_1,\ldots,U_{n_2}$ be a set of uniformly distributed random variables on $(0,1)$. Define, $U_{(m)}$ as the $m=\lceil (n_2+1)(1-\alpha) \rceil$-th order statistic of $U_1,\ldots,U_{n_2}\overset{iid}{\sim}Unif(0,1),$ with mean $\mu_\alpha$ and variance $\sigma_\alpha^2.$

Recall that, we have $a_{2n_2} = C\sqrt{\log(n_2)/n_2},$ for some constant $C>0$. Then, following similar calculations as in \eqref{eq: exp}, and \eqref{eq: r1n}, we have
\beq
    \begin{split}
         \E(\qa\mid\Do) = \E_1\tf^{-1}(U_{(m)}) 
     = \tf^{-1}(\mu_\alpha) + 0 + \frac{\sigma_\alpha^2}{2} \tf^{-1''}(\mu_\alpha) + R_{6n},   
    \end{split}
\eeq
where
\begin{equation*}
    \begin{split}
    R_{6n} & = -\E_1\bigg[ \tf^{-1}(\mu_\alpha) + (U_{(m)}-\mu_\alpha){\tf^{-1'}}(\mu_\alpha) \\&\hspace{4cm}+ \frac{1}{2} (U_{(m)}-\mu_\alpha)^2 {\tf^{-1''}}(\mu_\alpha) \bigg] \mathbbm{1}(|U_{(m)}-\mu_\alpha|>a_{2n_2}) \\
    & \hspace{1cm} + \frac{1}{2} \E_1(U_{(m)}-\mu_\alpha)^2 \left( \tf^{-1''}(\mu_*) - \tf^{-1''}(\mu_\alpha) \right) \mathbbm{1}(|U_{(m)}-\mu_\alpha|\leq a_{2n_2}) \\
    & \hspace{1cm} + \E_1\tf^{-1}(U_{(m)}) \mathbbm{1}(|U_{(m)}-\mu_\alpha|>a_{2n_2}),
    \end{split}
\end{equation*} 
where $\mu_*$ is some point between $U_{(m)}$ and $\mu_\alpha$ and we have by assumptions \ref{A1}$_1$ and \ref{A2}, $\E|R_{6n}|=o(1/n_2).$ Next, we show that the quantiles (conditional on $\Do$) of $|\hat\epsilon_i|$ are close to the quantiles of $|\epsilon_i|.$ Note that for some fixed $a$ ad $h,$ by Taylor's theorem
\beqs
\begin{split}
    \tf(a+h) 
    &= \hf(a+h) - \hf(-a-h)\\
    & = \fme(a) + h \fme'(a) + \E_1(\hg(X)) \{\fe'(a) - \fe'(-a)\} + R_{7n}(a,h),
\end{split}
\eeqs
where $\hf$ is the CDF of $\hat\epsilon_i$ conditional on $\Do$ and $R_{7n}(a,h)$ is a remainder term which can be bounded by following similar derivations in \eqref{eq: r2n}, \eqref{eq: prob_r2n} and using assumption \ref{A1}$_1.$
Therefore, by choosing $h = H_{|\epsilon|}(a) \E_1(\hg(X)) ,$ where $H_{|\epsilon|}(a) = -\{\fe'(a) - \fe'(-a)\}/\fme'(a)$ and $a = \fme^{-1}(\mu_\alpha),$ we obtain,
\beqs
\begin{split}
    \tf^{-1}(\mu_\alpha) = \fme^{-1}(\mu_\alpha) + H_{|\epsilon|}\Big(\fme^{-1}(\mu_\alpha)\Big) \E_1(\hg(X)) + R_{9n},
\end{split}
\eeqs
where $\E|R_{9n}|=o(n_1^{-\beta_1}).$ Therefore, we have
\beqs
\begin{split}
    \E(\qa) &= \E\E(\qa\mid\Do)\\ &= \fme^{-1}(\mu_\alpha) +   H_{|\epsilon|}\Big(\fme^{-1}(\mu_\alpha)\Big) \E\E_1(\hg(X)) + \frac{\sigma_\alpha^2}{2} {\fme^{-1}}''(\mu_\alpha) + o\bigg( \frac{1}{n_1^{\beta_1}}+\frac{1}{n_2}\bigg).
\end{split}
\eeqs
and by similar argument as in \eqref{eq: vare}, \eqref{eq: cvar} and \eqref{eq: var}, we get
\beqs
\begin{split}
    \Var(\qa) = \Big\{ H_{|\epsilon|}\Big(\fme^{-1}(\mu_\alpha)\Big) \Big\}^2 \Var\E_1(\hg(X)) +\sigma_\alpha^2 \Big{\{\fme^{-1}}'(\mu_\alpha)\Big\}^2 + o\bigg(\frac{1}{n_1^{\beta_2}}+\frac{1}{n_2}\bigg).
\end{split}
\eeqs

Note that, if the distribution of $\epsilon$ is symmetric, then we would have $H_{|\epsilon|}(a) = 0,$ for all $a$ and the new correction factor $h$ would be $h = G_{\epsilon}(a) \E_1(\hg^2(X))$ where $G_{\epsilon}(a) = -\fe''(a)/(2\fe'(a)),$ and we would get

\beqs
\begin{split}
    \E(\qa) &= \fme^{-1}(\mu_\alpha) +  G_{\epsilon}\Big(\fme^{-1}(\mu_\alpha)\Big) \E\E_1(\hg^2(X)) + \frac{\sigma_\alpha^2}{2} {\fme^{-1}}''(\mu_\alpha) + o\bigg( \frac{1}{n_1^{\beta_2}}+\frac{1}{n_2}\bigg).\\
    \Var(\qa) &= \Big\{ G_{\epsilon}\Big(\fme^{-1}(\mu_\alpha)\Big) \Big\}^2 \Var\E_1(\hg^2(X)) +\sigma_\alpha^2 \Big{\{\fme^{-1}}'(\mu_\alpha)\Big\}^2 + o\bigg(\frac{1}{n_1^{\beta_4}}+\frac{1}{n_2}\bigg).
\end{split}
\eeqs

Finally, it follows that $\E(\La) = 2\E(\qa)$ and $\Var(\La) = 4\Var(\qa)$.
\end{proof}

\bibliographystyle{plainnat}
\bibliography{ref}

\end{document}